\documentclass[reqno,11pt]{amsart}
\usepackage{tabularx}
\usepackage{amsmath, amsthm}
\usepackage{amssymb}
\usepackage{enumitem}
\usepackage{xypic}
\usepackage{mathrsfs}
\usepackage{hyperref}
\usepackage{graphicx}
\usepackage{perpage}
\usepackage{extarrows}
\MakePerPage{footnote}
\usepackage{tikz}
\usepackage{setspace}
\date{\today}
\usepackage{longtable}
\hypersetup{colorlinks,linkcolor={blue},citecolor={blue},urlcolor={red}}  
\allowdisplaybreaks          
\usepackage[usenames,dvipsnames]{pstricks}
\usepackage[shell]{pdftricks}
\begin{psinputs}
\usepackage{pst-grad} % For gradients
\usepackage{pst-plot} % For axes
\usepackage{epsfig}
\end{psinputs}

\title{Classification of Lipschitz simple function germs}

\author{Nhan Nguyen, Maria Ruas and Saurabh Trivedi}

\address{Departamento de Matem\'atica, ICMC - Universidade de S\~ao Paulo, 134560-970, S\~ao Carlos, S.P, Brasil}
\email{nguyenxuanvietnhan@gmail.com}
\email{maasruas@icmc.usp.br}
\email{saurabh.trivedi@gmail.com}

\newtheorem{thm}{Theorem}[section]
\newtheorem{lem}[thm]{Lemma}
\theoremstyle{theorem}

\newtheorem{cor}[thm]{Corollary}
\newtheorem{example}[thm]{Example}

\newcommand{\bb}{\mathbb}
\newcommand{\al}{\mathcal}
\newcommand{\ak}{\mathfrak}
\newcommand{\fs}{\mathscr}
\newcommand{\bff}{\mathbf}

\newcommand{\lct}{{\rm lct}}
\newcommand{\supp}{{\rm supp}}

\theoremstyle{remark}
\newtheorem{rem}[thm]{Remark}
\begin{document}
\maketitle
\begin{abstract} It is known that the bi-Lipschitz right classification of function germs admit moduli. In this article we introduce a notion called the Lipschitz simple function germs and present a full classification in the complex case. A surprising consequence of our result is that  a function germ is Lipschitz modal if and only if it deforms to the smooth unimodal family of singularities called $J_{10}$ in Arnold's list.
\end{abstract}

\section{Introduction}

After Thom and Mather's foundational works on classification of elementary catastrophes and singularities of mappings, Arnold \cite{Arnold1} presented a complete classification of simple, unimodal and bimodal function germs under the right equivalence. The simple singularities include two series $A_k, k\geq 1, D_k, k \geq 4$ and three exceptional germs called $E_6, E_7, E_8$, together called the $ADE$-singularities. Among the modal germs the one with the lowest codimension is $T_{3,3,3}$ of corank $3$ and codimension $8$. For the corank 2 modal germs $X_9$ with codimension 9 has the lowest codimension. In positive characteristics, the simple singularities were classified recently by Greuel--Nguyen \cite{gn} and unimodal and bimodal by Nguyen \cite{Duc}. 

It is well known from a result of Mostowski \cite{Mostowski} that the bi-Lipschitz right equivalence of complex analytic set germs does not admit moduli. This result was extended by Parusi\'nski \cite{Parusinski} for subanalytic sets and by Valette and the first author \cite{nv} for definable sets in polynomially bounded o-minimal structures. However, this is not true for function germs. Henry and Parusi\'nski \cite{hp1} showed that the bi-Lipschitz right classification of function germs does admit moduli. They constructed bi-Lipschitz invariants that vary continuously in the family
$f_{\lambda}(x,y) = x^3 - \lambda^2 x y^4 + y^6$ of corank $2$ function germs. This led us to ask whether there exists a classification for the bi-Lipschitz equivalence of function germs analogous to the smooth case. 

In this article we introduce the notion of Lipschitz simple function germs and give a complete classification of Lipschitz simple singularities. Roughly speaking, a finitely determined germ $f$ is said to be Lipschitz simple if there is a neighbourhood around a sufficiently high $k$-jet of $f$ that intersects only finitely many bi-Lipschitz equivalence classes. A germ is said to be modal if it is not simple. Notice that a smooth simple germ is also Lipschitz simple and so the $ADE$-singularities are Lipschitz simple. We observe that a germ is Lipschitz modal if it deforms to a Lipschitz modal family. This observation provides us with a strategy to list all Lipschitz simple germs which we describe below.

In Section \ref{section4} we prove that the rank and corank of function germs are Lipschitz invariants. We would like to emphasize that these seemingly simple results about Lipschitz equivalence were not known before and their proofs use some very recent work of Sampaio \cite{Sampaio}. 

While bi-Lipschitz invariance of rank and corank allows us to distinguish between the bi-Lipschitz type of certain germs, there are several germs in Arnold's list with same corank and codimensions that, a priori, could be bi-Lipschitz equivalent. In Section \ref{sec5}, we prove that the multiplicity and the singular set of  the algebraic tangent cone of the non-quadratic part of a germ obtained after applying the splitting lemma are also bi-Lipschitz invariants. This allows us to differentiate the bi-Lipschitz type of many germs. However, it is still not enough to write the Lipschitz normal forms as in Section \ref{sec9}. The only pair left in our analysis is $(Q_{11},S_{11})$. This pair is shown to be of different bi-Lipschitz type by using a result of Bivi\`a-Ausina and Fukui \cite{Fui} on the bi-Lipschitz invariance of the log canonical threshold of ideals.

We show in Section \ref{section7} that $J_{10}: x^3 + t xy^4 + y^6 + Q$ where $Q$ is a non-degenerate quadratic form,  is Lipschitz modal. We would like to remark that the result does not immediately follow from Henry--Parusi\'nski's \cite{hp1} result. This result implies that any germ that deforms to $J_{10}$ is Lipschitz modal. It is worth mentioning that the result follows trivially if there were a Lipschitz version of Splitting lemma. Unfortunately, this is not known. We observe later that $J_{10}$ is the Lipschitz modal singularity of the smallest codimension. 

The main result is contained in Section \ref{sec9} where we present a complete list of Lipschitz simple germs (Theorem \ref{thm_corank2} and Theorem \ref{thm_corank3}). This uses the fact that most singularities in Arnold's list deform to $J_{10}$ (see Lemma \ref{lem_deform_j10}). For example, all smooth bimodal singularities deform to $J_{10}$. Therefore, they are Lipschitz modal. The rest of the singularities are shown to be Lipschitz simple by proving that they belong to a bi-Lipschitz trivial family (see Section \ref{section8}) and are not adjacent to any class of Lipschitz modal singularities. The bi-Lipschitz triviality is proved by  constructing vector fields satisfying the Thom-Levine criterion. Our method is inspired by and improves upon the works of Abderrahmane \cite{Abder}, Fernandes and Ruas \cite{fr}, Saia et al. \cite{Saiaetal} and Humberto and Fernandes \cite{Humbertof}. The non-adjacencies follow from a result of Brieskorn \cite{Brieskorn}. 

We would like to remark that in the paper we only consider the Lipschitz classification of complex simple germs. The real case turned out to be more delicate and requires more careful analysis. For example, in the real case the family $J_{10}$ is Lipschitz modal if $t \leq 0$ (see Section \ref{sec5}) and is Lipschitz simple when $t>0$ (see \cite{hp2}), and this must be taken into account while checking deformations. Moreover, the Milnor number is not a bi-Lipschitz invariant in the real case; see example in Section \ref{sec2}.

Throughout the paper,  $\|.\|$ denotes the Euclidean norm, and $|.|$ denotes the absolute value of real numbers or the modulus of complex numbers. Let $\bb K = \bb C$ or $ \bb R$, and let $f, g : \bb K^n \supset U \to \bb K$ be two functions. We write $f \lesssim g$ (or $g \gtrsim  f$) if there is a $C>0$ such that $|f(x)|\leq C |g(x)|$ for all $x \in U$. We write $f \sim g$ if $f \lesssim g$ and $g \lesssim f$.  Suppose that $U$ is a neighbourhood of a point $x_0$, we write $f = o(g)$  or $f \ll g$ at $x_0$ if $\lim_{x\to x_0} \frac{|f(x)|}{|g(x)|} = 0$. And, $f = O(g)$ or $f \gtrsim g$ at $x_0$ if $\lim_{x\to x_0} \frac{|f(x)|}{|g(x)|}$ is bounded.

 \section{Preliminaries} \label{sec2}

Let $\bb K = \bb R$ or $\bb C$. Given a smooth germ $f: (\bb K^n, 0) \to (\bb K, 0)$, we denote by $Df(x)$ the derivative of $f$ at the point $x$, by $V_f$ the zero set of $f$, and by $\Sigma_f= \{ x \in \bb K^n: Df(x) =0\}$ the \textit{singular set} of $f$. We write the Taylor expansion of $f$ at $0$ as $T_0 f = f_k + f_{k+1} + \ldots$,  where $f_i$ is the homogeneous polynomial of degree $i$, $f_k \neq 0$. We denote by $H_f = f_k$ the \textit{lowest degree homogeneous polynomial} and by $m_f = k$ \textit{the multiplicity of $f$ at $0$}. 

Given a semialgebraic set $X \in \bb K^n$, for $x \in \overline{X}$ where $\overline{X}$ denotes the closure of $X$, the tangent cone of $X$ at the point $x$ is defined as follows
$$ C( X, x)= \{ \lambda u: \lambda \in \bb R_{\geq 0}, u = \lim_{m\to \infty} \frac{x_m -x}{\|x_m-x\|}, X \supset \{ x_m\} \to x\}.$$

We call $C_\ak g(f,0)= C(V_f, 0)$ the geometric tangent cone of $f$. The algebraic tangent cone $C_\ak a(f,0)$ of $f$ at $0$ is defined by $$C_\ak a(f,0) = \{x \in \bb K^n: H_f(x)=0\}.$$
By the \textit{singular set of the algebraic tangent cone of $f$} we mean the set $\Sigma_{H_f}$. 

For $\bb K = \bb C$, it is known that $C_\ak a(f,0) =C_\ak g(f, 0)$ (see Theorem 10.6 in Whitney \cite{Whitney} or, Proposition 2.6 in \cite{hll}). This is not true in the real case. For instance, if $f(x,y) = x^2 +y^{2k}$ where $k$ is an integer greater than or equal to $2$, then $C_\ak a(f,0) = \{x = 0\}$ and $ C_\ak g(f, 0) =\{(0,0)\}$.

Two smooth function germs $f,g : (\bb K^n, 0) \to (\bb K, 0)$ are said to be \textit{bi-Lipschitz right equivalent} if there exists a germ of bi-Lipschitz homeomorphism  $\varphi : (\bb K^n,0) \to (\bb K^n,0)$ such that $f\circ \varphi = g$. More precisely, there exist a neighbourhood $U$ of $0 \in \bb K^n$, a homeomorphism $\varphi : U \to \varphi(U)$ with $\varphi(0) = 0$ and
\begin{equation}\label{equ_lipschitz_const}
\frac{1}{L}\|x - y\| \leq \|\varphi(x)-\varphi(y)\| \leq L\|x-y\|\
\end{equation} for any $x,y \in U$ and for some $L$ positive such that $f \circ \varphi = g$. The least $L$ satisfying (\ref{equ_lipschitz_const}) is said to be the \textit{bi-Lipschitz constant} of $\varphi$. In the rest of the article by \textit{``equivalence"} we mean the right equivalence.

It follows from the main result of Sampaio \cite{Sampaio} that if $f$ and $g$ are bi-Lipschitz equivalent then $C_\ak g(f, 0)$ and $C_\ak g(g, 0)$ are also bi-Lipschitz homeomorphic as sets.

We denote by $\al E_n$ the set of all smooth function germs at $0 \in \bb K^n$ and by  $\ak m_n$ the set of germs in $\al E_n$ vanishing at $0$. Notice that $\al E_n$ is a local ring and $\ak m_n$ is the maximal ideal of $\al E_n$. Denote by $\al R_n$ the group of all smooth diffeomorphism germs $H : (\bb K^n,0) \to (\bb K^n,0)$. Then $\al R_n$ acts on $\al E_n$ by composition $H.f = f \circ H$. Two germs $f,g \in \al E_n$ are called \textit{smoothly equivalent}, denoted by $f \sim_{\al R} g$, if they lie in the same orbit of this action.

Given a germ $f \in \al E_n$, denote by $Jf$ the Jacobean ideal of $f$, i.e. the ideal in $\al E_n$ generated by the partial derivatives of $f$. The \textit{codimension} (also called the \textit{Milnor number} of $f$) of a germ $f \in \al E_n$ is defined to be $\dim_{\bb K} \al E_n/Jf$. It is well-known that in the complex case the Milnor number is a topological invariant (Milnor \cite{Milnor1}), and is therefore a bi-Lipschitz invariant, i.e. if $f,g \in \al E_n$ are bi-Lipschitz equivalent then their codimensions are equal.  We would like to remark that the Milnor number is not a bi-Lipschitz invariant in the real case. Consider the family $f_t (x, y ) = x^4 + y^4  + tx^2y^2 +y^6$, $t \geq 0$. By Remark \ref{rem_fr_thm}, this family is bi-Lipschitz trivial, however, $\mu(f_t) =9, t\neq 2$ and $\mu(f_2) = 13$. 

Denote by $J^k(n,1)$  the  $k$-jet space of $\ak m_n$, by $\al R^k_n$ the set of $k$-jet space of $\al R_n$. A germ $f \in \ak m_n$ is called \textit{$k$-determined} if for any $g \in  \ak m_n$ such that $j^k g(0) = j^kf (0)$ then $f \sim_{\al R} g$; $f$ is called \textit{finitely determined} if $f$ is $k$-determined for some $k \in \bb N$. Recall that if a germ $f : (\bb K^n,0) \to (\bb K,0)$ is finitely determined then it has an isolated singularity at $0$. In addition, if a germ $f$ is finitely determined, then for a sufficiently large $k$ there is a neighbourhood of $j^k f(0)$ in $J^k(n,1)$ such that every germ in this neighbourhood is $k$-determined. We call the number $k$ \textit{sufficiently large} for $f$.

The action of $\al R^k_n$ on $J^k(n,1)$ is defined by taking the composition and then truncating  the Taylor expansion. A germ $f \in \ak m_n$ is said to be \textit{smooth simple} if for $k\in \bb N$ sufficiently large for $f$, there exist a neighbourhood of $j^kf(0)$ in $J^k(n,1)$ that meets only finitely many $\al R^k
_n$-orbits in $J^k(n,1)$. The germs that are not smooth simple are called \emph{smooth modal}.  

Two $k$-jets $z,w \in J_0^k(n,1)$ are \textit{bi-Lipschitz equivalent} if there exists a bi-Lipschitz homeomorphism germ $\phi : (\bb K^n,0) \to (\bb K^n,0)$ such that $z\circ \phi= w$. By Lipschitz orbits we mean the equivalence classes of the bi-Lipschitz equivalence.  A germ $f \in \ak m_n$ is \emph{Lipschitz simple} if for $k\in \bb N$  sufficiently large for $f$, there is a neighbourhood of $j^kf(0)$ in $J^k(n,1)$ that meets only finitely many Lipschitz orbits. The germs that are not Lipschitz simple are called \emph{Lipschitz modal}. We remark that if a germ is smoothly simple then it is also Lipschitz simple. The converse is not true in general. 

Given $f \in \al E_n$, the \textit{corank} of $f$ at $0$ is defined to be the nullity (dimension of the kernel) of the Hessian $\left ( \frac{\partial^2 f}{\partial x_i \partial x_j}(0)\right )$.  The following result is known as the Splitting lemma; see  Ebeling \cite{Ebeling} or Gibson \cite{Gibson} for example.

\begin{lem}\label{splitting}
	Let $f \in \ak m_n^2$ be a finitely determined germ of corank $c$. Then there exists $g \in \ak m_c^3$ such that 
	$$f(x_1,\ldots,x_n) \sim_{\al R} g(x_1,\ldots,x_c) \pm x_{c+1}^2 \pm \ldots \pm x_n^2.$$
The codimension of $f$ and $g$ are equal. Moreover, $g$ is uniquely determined up to smooth equivalence, i.e. if $g + Q \sim_{\al R} h + Q$ then $g \sim_{\al R} h$, where $Q$ is the non-degenerate quadratic form.	
\end{lem}

\section{Arnold's results on classification}\label{sec3}
In this section, we recall some results of Arnold about the classification of complex analytic germs that will be used in the later sections. 

It is known that the germs in $\ak m_n$ of codimension $0$ (or non-singular germs) are equivalent to the germ $(x_1 \ldots, x_n)\mapsto x_1$, hence they are smooth simple. Arnold classified germs in $\ak m_n^2$ with isolated singularity 
into classes of constant Milnor numbers, represented by a germ or family of germs called the `normal form' of the class of singularities. More precisely, a class $\al D$ has a \textit{normal form} $\ak F_t(x_1, \ldots, x_k)$, if $\al D$ contains all germs smoothly equivalent to $\ak F_t + x_{k+1}^2 + \dots + x_n^2$ for some $t$. 

Recall that a smooth germ $F: (\bb C^n \times \bb C^k, 0) \to  (\bb C, 0)$, $F(x, t)=f_t(x)$ is said to be a $k$-parameter deformation of a germ $f: (\bb C^n, 0) \to (\bb C, 0)$ if $f_0(x)=  f(x)$.  A class of singularities $\al C$ is said to be \textit{adjacent} to a class $\al D$,  denoted $\al C \rightarrow \al D$, if every $f \in  \al C$ can \textit{deform} to $\al D$. That is, for every $f \in \al C$, there exists a smooth deformation $f_t$ of $f$ such that $f_0=f$ and $f_t$ belongs to $ \al D$ for every $t$ near $0$. 

The following results are due to Arnold \cite{Arnold1}.

\begin{thm}\label{thm_arnold1}
A germ $f \in \ak m^2_n$ is smooth simple if and only if it is equivalent to one of the germs in the following table:
\begin{center}
\footnotesize
\begin{tabular}{c|c | c}
\hline 
Name & Normal form & Codimension \\ \hline
$A_k$ & $x^{k+1}$ & $k \geq 1$\\
$D_k$ & $x^2y + y^{k-1}$ & $k\geq 4$\\
$E_6$ & $x^3 + y^4$ & 6\\
$E_7$ & $x^3 + xy^3$  & 7\\
$E_8$ & $x^3 + y^5 $ & 8\\  \hline
\end{tabular}
\end{center}
\end{thm}

\begin{thm}\label{thm_arnold2}
The adjacencies of all smooth modal germs can be described by the diagrams below. The modal germs that do not appear in the diagram deform to one of the classes inside the brackets. 

\textbf{Corank $2$:}
$$J_{10} = (T_{2,3,6}) \longleftarrow \cdots$$

\[
\xymatrixrowsep{0.2in}
\xymatrixcolsep{0.2in}
\xymatrix{
& & X_9 = T_{2,4,4}\\
\cdots \ar[r]& (T_{2,4,6})\ar[r] & T_{2,4,5}\ar[u] & Z_{11} \ar[l] & (Z_{12})\ar[l] & \cdots \ar[l]\\
\cdots \ar[r]& (T_{2,5,6}) \ar[u]\ar[r] & T_{2,5,5}\ar[u] & W_{12}\ar[u] \ar[l] & (W_{13}) \ar[u]\ar[l] & \cdots \ar[l]
}
\]

\textbf{ Corank $3$:}

\[
\xymatrixrowsep{0.2in}
\xymatrixcolsep{0.1in}
\xymatrix{
 & & & &  P_8 = T_{3,3,3} & & & \\
& \cdots \ar[r]& (T_{3,3,6}) \ar[r] &T_{3,3,5} \ar[r]& T_{3,3,4} \ar[u] & Q_{10}\ar[l] & Q_{11} \ar[l] &(Q_{12}) \ar[l] &  \cdots \ar[l] &\\
& \cdots \ar[r] & (T_{3,4,6}) \ar[dr] & \\
\cdots \ar[r] & (T_{3,5,6}) \ar[r] & T_{3,5,5} \ar[r] & T_{3,4,5} \ar[uu] \ar[r] & T_{3,4,4} \ar[uu] & S_{11} \ar[l]  \ar[uu] & S_{12} \ar[l]  \ar[uu] & (S_{1,0}) \ar[l] & \cdots \ar[l]\\
\cdots \ar[r] & (T_{4,5,6})\ar[u]\ar[r] & T_{4,5,5} \ar[u] \ar[r] & T_{4,4,5} \ar[u] \ar[r] & T_{4,4,4} \ar[u] & (U_{12}) \ar[l] \ar[u] & \ar[l] \cdots &  \\
\cdots \ar[r] & (T_{5,5,6}) \ar[u] \ar[r] & T_{5,5,5}  \ar[u]  & (T_{4,4,6}) \ar[u] & (\bff{O}) \ar[u] &  \\
& & & \vdots \ar[u] & \vdots \ar[u] &  
}
\]
where the normal forms of the above singularities are

\begin{center}
\begin{longtable}{l l l}
$J_{10}$& $x^3 +ax^2y^2 +y^6$ &   $4a^3+27\neq 0$\\
$X_9$ & $x^4 +y^4 +ax^2y^2$ &   $a^2 \neq 4$\\
$P_8$ & $x^3 + y^3 + z^3 + a xyz$, & $ a^3+27\neq  0$\\

$T_{p, q, r}$ & $x^p + y^q + z^r + a xyz$,& $ a\neq 0, \frac{1}{p} + \frac{1}{q} + \frac{1}{r} <1$ \\
$Z_{11}$& $x^3y +y^5 +axy^4$ & \\
$Z_{12}$& $x^3y +xy^4 +ax^2y^3$ &\\
$W_{12}$& $x^4 +y^5 + ax^2y^3$&\\
$W_{13}$& $x^4 +xy^4 + ay^6$ & \\
$Q_{10}$& $x^3 + y^4 + yz^2 + axy^3$ &\\	 
$Q_{11}$& $x^3 + y^2z + xz^3 + az^5$& \\
$Q_{12}$& $x^3 + y^5 +yz^2 + a xy^4$&\\		
$S_{11}$& $x^4 + y^2z + xz^2 + ax^3z$& \\
$S_{12}$& $x^2y + y^2z + xz^3 + az^5$ &\\
$U_{12}$& $x^3 + y^3 + z^4 + a xyz^2$ &\\
$S_{1,0}$&$x^2z+ yz^2 +y^5 + a zy^3 + b zy^4$&\\
$\bff{O}$& $\text{ germs of corank } \geq 4$ & 
\end{longtable}
\end{center}
\end{thm}
Notice that by changing $x\mapsto x + \lambda y^2$, the normal form of $J_{10}$ becomes $x^3 + a xy^4 +y^6$. This is exactly the family considered by Henry--Parusi\'nski in \cite{hp1}. For convenience, we will take $x^3 + a xy^4 +y^6$ as the normal form of $J_{10}$.

\section{Bi-Lipschitz invariance of rank and corank}\label{section4}

The aim of this section is to prove that the rank and corank of smooth germs are bi-Lipschitz invariants. We use a construction of a map associated to a bi-Lipschitz map $\varphi$ due to Sampaio \cite{Sampaio}, which in some sense works as the derivative of $\varphi$. Let us give the details of the construction.

Let $\varphi : (\bb K^n,0) \to (\bb K^n,0)$ be a bi-Lipschitz map germ with $\psi$ as its inverse. By definition, there exist $r,L \in \bb R^+$ with $L \geq 1$ such that for all $x,y \in \bff B(0,r)$, where $\bff B(0,r)$ denotes the closed ball of radius $r$  centered at $0$, we have:
$$\frac{1}{L}\|x - y\| \leq \|\varphi(x)-\varphi(y)\| \leq L\|x-y\|.$$
For $m \in \bb N$, put $$\varphi_m(x) = m\varphi(\frac{x}{m})\,\,\, \text{ and } \,\,\,\psi_m(x) = m\psi(\frac{x}{m}).$$ It is easily checked that $\varphi_m$ and $\psi_m$ are a bi-Lipschitz germs of the same Lipschitz constant for any $m$. By the Arzela--Ascoli theorem there exists a subsequence $\{m_i\} \subset \bb N$  such that $\{\varphi_{m_i}\}$ and $\{\psi_{m_i}\}$ converges uniformly to a bi-Lipschitz germs $d\varphi$ and $d\psi$ respectively. It was proved in \cite{Sampaio} that $d\psi$ is the inverse of $d\varphi$. Notice that  the domains of $\varphi_{m_i}$ and $\psi_{m_i}$ grow bigger as $m_j$ increases and the limits $d\varphi$ and $d\psi$ are well-defined on the whole of $\bb K^n$. 

We first prove the bi-Lipschitz invariance of the rank.

\begin{thm}\label{rank}
Let $f,g : (\bb K^n, 0) \to (\bb K^p,0)$ be two smooth map germs. If $f$ and $g$ are bi-Lipschitz equivalent then the rank of $f$ is equal to the rank of $g$ at $0$. 
\end{thm}

\begin{proof}
Since $f$ and $g$ are bi-Lipschitz equivalent, there exists a germ of a bi-Lipschitz homeomorphism $\varphi : (\bb K^n, 0) \to (\bb K^n,0)$ such that $f \circ  \varphi = g$. Then, there exists a sequence $\{m_i\}$ in $\bb N$ such that the sequence of maps $\varphi_{m_i} : (\bb K^n,0) \to (\bb K^n,0)$ defined by $ \varphi_{m_i} = m_i \varphi(\frac{x}{m_i})$ converges uniformly to a bi-Lipschitz germ $d\varphi$ on a neighbourhood of $0$  as $i$ tends to $\infty$.  Write $f(x) = Df(0)(x) + o(\|x\|)$ and $g(y) = Dg(0)(y) + o(\|y\|)$ on a neighbourhood of $0$. Since $f \circ \varphi(y) = g(y)$
\begin{alignat*}{3}
& &&\qquad \qquad \qquad \qquad \qquad\quad\quad\,\,\, m_i f \circ \varphi(\frac{y}{m_i}) &&= m_i g(\frac{y}{m_i})\\  
&\implies &&\qquad \qquad\,\,\,m_i(Df(0)( \varphi(\frac{y}{m_i})) + m_i o\|\varphi(\frac{y}{m_i})\|   &&= m_i Dg(0)(\frac{y}{m_i}) + m_i o\|\frac{y}{m_i}\| \\
&\implies &&\,\,\, \lim_{i \to \infty}  Df(0)(m_i\varphi(\frac{y}{m_i})) + \lim_{i \to \infty} m_i o\|\varphi(\frac{y}{m_i})\|  &&= Dg(0)(y) + \lim_{i \to \infty} m_i o\|\frac{y}{m_i}\|.\\
\end{alignat*}
Hence,  $$Df(0)(d\phi(y))  = Dg(0)(y)$$
for $y$ in a small neighbourhood of $0$. 
This implies that the rank of $f$ and $g$ are equal.
\end{proof}

In the rest of the section we fix $f , g : (\bb K^n,0) \to (\bb K,0)$ to be bi-Lipschitz equivalent smooth function germs, and  fix $\varphi : (\bb K^n,0) \to (\bb K^n,0)$ to be a germ of a bi-Lipschitz homeomorphism such that $f \circ \varphi = g$.
   
To show that the corank of finitely determined smooth germs is also a bi-Lipschitz invariant, we need the following lemmas. It has been known that the multiplicity is a bi-Lipschitz invariant (see \cite{Risler}, \cite{fr}). We present another proof below. 

\begin{lem}\label{lem_multiplicity} The multiplicities $m_f$ and $m_g$ of $f$ and $g$ are equal. 
\end{lem}
\begin{proof}
Write $\varphi = (\varphi_1, \ldots, \varphi_n)$ where $\varphi_i:(\bb K^n, 0) \to (\bb K, 0)$. Assume that $m_f = k$ and $m_g=l$. In a neighbourhood of $0$ we can write
$$ f(x) = H_f + o(\|x\|^k) \text{ and } g(x) = H_g + o(\|x\|^l).$$
By a change of coordinates, we may assume that $H_f(1, 0, \ldots, 0) \neq 0$ and $H_g(1, 0, \ldots,0) \neq 0$. Then, we can write $g$ in two ways:
$$ g (x) = x_1^l +\ldots +o(\|x\|^l)$$
and $$ g(x) = f(\varphi(x)) = \varphi_1^k(x) + \ldots+ o(\|x\|^k).$$
Since $\varphi$ is bi-Lipschitz, each $\varphi_i$, $i=1, \ldots, n$ is Lipschitz. Thus, 
$$\|\varphi_1^k(x) + \ldots+ o(\|x\|^k) \| \lesssim \|x\|^k.$$ Restricting to the $x_1$-axis, we have
$$1=\lim_{x\to 0} \frac{\|\varphi_1^k(x) + \ldots+ o(\|x\|^k)\|}{\| x_1^l +\ldots +o(\|x\|^l)\|}\lesssim \lim_{x_1\to 0}\frac{|x_1|^k}{|x_1|^l} =\lim_{x_1\to 0}|x_1|^{k-l} .$$
This implies that $l\geq k$. Interchanging $f$ with $g$, we have also $l\leq k$. Therefore, $k=l$.
\end{proof}

\begin{lem}\label{lem_cone_equ} The lowest degree homogeneous parts $H_f$ and $H_g$ of $f$ and $g$  are bi-Lipschitz equivalent. 
\end{lem}
\begin{proof}
We can write  $f$ and $g$ at $0$ as
\begin{equation}\label{eq_lem_cone}
f(x) = H_f(x) + o(\|x\|^{k}) \text{ and } g(x) = H_g(x) + o(\|x\|^{k})
\end{equation} 
Then, for $x \in \bb K^n$ and $m \in \bb N$ big enough we have
$$m^k \left [f (\varphi  (\frac{x}{m} ))\right ] =m^kg (\frac{x}{m}).$$
Substituting the expansions of $f$ and $g$ in the equation (\ref{eq_lem_cone}), we have
$$m^k \left [H_f (\varphi(\frac{x}{m}) ) + o(\|\varphi(\frac{x}{m})\|^{k})\right ]  = m^k\left [H_g (\frac{x}{m}) + o \left(\|\frac{x}{m}\|^{k}\right)\right ] 
$$
Since the multiplicity is a bi-Lipschitz invariant, $H_f$ and $H_g$ are homogeneous polynomials of the same degree $k = m_f = m_g$. Thus,
$$H_f(m \varphi(\frac{x}{m})) + m^k o(\|\varphi(\frac{x}{m})\|^{k}) = H_g(x) + m^k  o\left(\|\frac{x}{m}\|^{k}\right).
$$
Then, by the Arzela--Ascoli theorem, there exists a sequence $\{m_i\}$ of positive integers such that as $i$ tends to $\infty$, the sequence of maps $\{m_i\varphi(\frac{x}{m_i})\}$ converges to a bi-Lipschitz map germ $d\varphi$. By taking limits on both sides of the above equation as $m_i$ tends to $\infty$, we get
$$H_f(d\varphi(x)) = H_g(x).$$ This completes the proof.
\end{proof}

\begin{rem}\label{rem_con_equ} (i) Consider the sequence $\{m_i\}$ associated with the map $\varphi$ as in the proof of Lemma \ref{lem_cone_equ}.  Let $f: \bb K^n, 0 \to \bb K,0$ be a smooth function germ with the multiplicity $m_f = k$. We always have
 $$\lim_{i\to \infty} m_i^{k} f(\varphi(\frac{x}{m_i})) \to H_f(d\varphi(x)).$$

(ii) The result in Lemma \ref{lem_cone_equ} is sharp. In general, $T_0^l(f)$ is not bi-Lipschitz equivalent to $T_0^l(g)$ for $l>k$, where $T_0^l(f)$ denotes the Taylor expansion of $f$ at $0$ up to degree $l$. For example, consider $f(x, y) = x^3 +y^6$ and $g(x, y) = (x +y^2)^3 +y^6$. It is easy to see that $f$ and $g$ are smooth equivalent under the change of coordinates $\varphi(x,y) = (x +y^2,y)$. However, $T_0^4(f) = x^3$ is not bi-Lipschitz equivalent to $T_0^4(g) = x^3 + 3x^2y^2$. 
\end{rem}

\begin{lem}\label{lem_invariant_derivative} Let $L$ be the bi-Lipschitz constant of $\varphi$. Then,
$$L^{-1}\|Df(\varphi(x))\| \leq \|D g(x))\| \leq L \|D f(\varphi(x))\|.$$
\end{lem}

\begin{proof} The proof follows directly from the definition of the derivative. First,	
\begin{align*}
D g(x) &= \lim_{\Delta x\to 0} \frac{g(x + \Delta x) -g(x)}{\|\Delta x\|}\\
&= \lim_{\Delta x\to 0} \frac{f\circ \varphi(x + \Delta x) -f\circ \varphi(x)}{\|\Delta x\|}\\
&=  \lim_{\Delta x\to 0} \frac{f(\varphi(x) + \varphi(x + \Delta x) - \varphi(x)) -f(\varphi(x))}{\|\varphi(x + \Delta x) - \varphi(x)\|} \frac{\|\varphi(x + \Delta x) - \varphi(x)\|}{\|\Delta x\|}\\
\end{align*}
Since $\Delta x\to 0$ implies that $\varphi(x + \Delta x) - \varphi(x) \to 0$, we have
$$\lim_{\Delta x\to 0} \frac{f(\varphi(x) + \varphi(x + \Delta x) - \varphi(x)) -f(\varphi(x))}{\|\varphi(x + \Delta x) - \varphi(x)\|}  = D f(\varphi(x)).$$
On the other hand, 
$$L^{-1} \leq \bigg\|\frac{\varphi(x +\Delta x) -\varphi(x)}{\Delta x}\bigg \| \leq L.$$
Therefore, 
$$L^{-1}\|D f(\varphi(x))\| \leq \|D g(x)\| \leq L \| D f(\varphi(x))\|.$$
\end{proof}

\begin{lem}\label{lem_singlocus} $\Sigma_{f}$ and $\Sigma_{g}$ are bi-Lipschitz equivalent.
\end{lem}
\begin{proof} By Lemma \ref{lem_invariant_derivative}, we have 
$$ \|Df(\varphi(x))\| \sim \|Dg(x)\| \text{ and } \|Df(x)\| \sim \|Dg(\varphi^{-1}(x)) \|.$$
This shows that $\varphi$ maps $\Sigma_{g}$ onto the $\Sigma_{f}$.
\end{proof}

\begin{thm} The corank of $f$ is equal to the corank of $g$ at $0$.
\end{thm}

\begin{proof}
If the corank of $f$ is $n$, then $m_f \geq 3$. We know from Lemma \ref{lem_multiplicity} that the multiplicity is a bi-Lipschitz invariant, so $m_g \geq 3$. This implies that the corank of $g$ is $n$.

Now we assume that the corank of $f$ is $r$ and that of $g$ is $k$ with $(0 \leq r < k < n)$. By the Splitting lemma, $f \sim_{\al R} R_f(x_1,\ldots,x_r) + Q_f$ and $g \sim_{\al R} R_g(x_1,\ldots,x_k) + Q_g$ where $Q_f$ and $Q_g$ are quadratic forms in further variables. 

By Lemma \ref{lem_cone_equ}, $Q_f$ and $Q_g$ are bi-Lipschitz equivalent. Then, their singular sets $\Sigma_{ Q_f}$ and $\Sigma_{ Q_g}$ are bi-Lipschitz equivalent (see Lemma \ref{lem_singlocus}). Notice that $\Sigma_{Q_f} = \bb K^{r} \times\{ 0\}$  and $\Sigma_{Q_g} = \bb K^{l} \times \{0\}$.  It is known that the dimension of a semialgebraic set is  a bi-Lipschitz invariant, therefore $r$ must be equal to $k$. This completes the proof.  
\end{proof}

\section{ Bi-Lipschitz type of the non-quadratic parts} \label{sec5}
Given a finitely determined germ $F : (\bb K^n \times \bb K^p,0) \to (\bb K,0)$ of corank $n$, we know by the splitting lemma that $F(x,y) \sim_{\al R} f(x) + Q_f(y)$ where $Q$ is a quadratic form in $p$ variables. Moreover, this splitting of $F$ is unique in the sense that if $F \sim_{\al R} g + Q_g$ then $f \sim_{\al R} g$. In this section we prove a weaker version of this result for bi-Lipschitz equivalence.
   
Consider two smooth germs $F, G: (\bb K^n \times \bb K^p, 0) \to (\bb K, 0)$ of the following forms:
$F(x, y) = f(x) + Q_f(y)$ and $G(x, y) = g(x) + Q_g(y)$ where 
$f, g : (\bb K^n,0) \to (\bb K, 0)$ are smooth germs of multiplicities $\geq 3$ and  $Q_f, Q_g$ are of the form $\pm y_1^2 \pm \ldots \pm y_p^2$. We prove the following result.

\begin{thm}\label{thm_splitting_weak}  If $F$ and $G$ are bi-Lipschitz equivalent, then

(i)  $m_f = m_g$;

(ii) $\Sigma_{H_f}$ and $\Sigma_{H_g}$ are bi-Lipschitz equivalent.
\end{thm}
\begin{proof}
Let $\varphi = (\varphi_1, \varphi_2): (\bb K^n \times \bb K^p, 0)\to (\bb K^n \times \bb K^p, 0)$ be a germ of a  bi-Lipschitz homeomorphism such that $F\circ \varphi = G$. Let $\psi = (\psi_1, \psi_2)$ be the inverse of $ \varphi$.  For $m \in \bb N$, set 
$$ \varphi_m = (\varphi_{1,m}, \varphi_{2,m}) = m\varphi\bigg(\frac{x}{m}, \frac{y}{m}\bigg)$$
and $$\psi_m = (\psi_{1,m}, \varphi_{2,m})  = m\psi\bigg(\frac{x}{m}, \frac{y}{m}\bigg).$$
We know that there exists a subsequence $\{m_i\}\subset \bb N$ such that $\varphi_{m_i}$ and $\psi_{m_i}$ uniformly converge to bi-Lipschitz homeomorphisms which we denote by $d\varphi$ and $d\psi$ respectively. Moreover, $d\psi$ is the inverse of $d\varphi$. In fact, $d\varphi = (d\varphi_1, d\varphi_2)$ and $d\psi = (d\psi_1, d\psi_2)$ where $d\varphi_1$ (reps. $ d\varphi_2, d\psi_1, d\psi_2$) is the limit of the sequence $\{\varphi_{1,m_i}\}$ (resp. $\{\varphi_{2,m_i}\}$, $\{\psi_{1,m_i}\}$, $\{\psi_{2,m_i}\}$).

By Lemma \ref{lem_cone_equ} and Lemma \ref{lem_singlocus}, we have $d\varphi(\Sigma_{Q_g}) = \Sigma_{Q_f}$. Since $\Sigma_{Q_g} = \Sigma_{Q_f} = \bb K^n \times \{0\}$, 
$$d\varphi (x, 0) = (d\varphi_1 (x, 0), 0).$$
The map $d\varphi$ is a bi-Lipschitz homeomorphism, so is the map $d\varphi_1(., 0): \bb K^n,0 \to \bb K^n,0$. It is easy to check that the inverse of $d\varphi_1(., 0)$ is the map $d\psi_1(., 0)$. 

By Lemma \ref{lem_invariant_derivative}, for all $(x, y) \in \bb K^n\times\bb K^p$ near $0$ we have

$$\|DF(\varphi(x, y))\| \lesssim \|DG(x, y)\|$$
and 
$$
\|DG(\psi(x, y))\| \lesssim \|DF(x, y)\|.$$
This yields that
\begin{equation}\label{equ_weak_spliting_1}
\|DF(\varphi(x, 0))\| \lesssim \|DG(x, 0)\|,
\end{equation}
and
\begin{equation}\label{equ_weak_spliting_2}
\|DG(\psi(x, 0))\| \lesssim \|DF(x, 0)\|.
\end{equation}
Notice that  $DF(\varphi(x, 0) = (Df(\varphi_1(x,0)), 2\varphi_2(x,0))$ and $DG(x,0) = Dg(x)$. The inequality (\ref{equ_weak_spliting_1}) implies
\begin{equation}\label{equ_week_splitting_3}
\|Df(\varphi_1(x, 0))\| \lesssim \|Dg(x)\|.
\end{equation} 
Similarly, (\ref{equ_weak_spliting_2}) implies
\begin{equation}\label{equ_week_splitting_4}
\|Dg(\psi_1(x, 0))\| \lesssim \|Df(x)\|.
\end{equation} 

We now prove (i). First we prove that $m_g \leq m_f$. Let us consider the subsequence $\{m_i\}$ mentioned at the beginning of the proof. It follows from (\ref{equ_week_splitting_3}) that
\begin{equation}\label{equ_week_splitting_5}
\|m_i^{m_g-1} Df(\varphi_1(\frac{x}{m_i}, 0))\| \lesssim \|m_i^{m_g -1}Dg(\frac{x}{m_i})\|.
\end{equation} 
The right-hand side of (\ref{equ_week_splitting_5}) tends to $D H_g(x)$ as $m_i \to \infty$ (see Remark \ref{rem_con_equ}, (i)), therefore, it is bounded when $x$ is bounded. Recall that $H_g$ is the lowest degree homogeneous polynomial of $g$. Let  $x^*$ be a point near $0$ such that $DH_f(d\varphi_1(x^*,0)) \neq 0$. If $ m_g > m_f$,  then $\|m_i^{m_g-1} Df(\varphi_1(\frac{x^*}{m_i}, 0))\| $ tends to infinity as $m_i \to \infty$, which violates (\ref{equ_week_splitting_3}). Thus, $m_g \leq m_f$. 
By the same arguments applied to (\ref{equ_week_splitting_4}), we get $m_g \geq m_f$. 

To prove (ii), it is enough to show that $d\varphi_1(., 0)(\Sigma_{H_f}) = H_g$, or equivalently, 
$$d\varphi_1(., 0)(\Sigma_{H_g}) \subset \Sigma_{H_f},  \text{ and } d\psi_1(., 0)(\Sigma_{H_f}) \subset \Sigma_{H_g}.$$

Since $m_f = m_g$, the left-hand side of (\ref{equ_week_splitting_5}) tends to $\|DH_f (d\varphi_1(x, 0))\|$ as $m_i \to \infty$. Thus,
$$ \|DH_f (d\varphi_1(x, 0))\| \lesssim \| DH_g(x)\|.$$
This shows that $ d\varphi_1(., 0)$ maps $\Sigma_{H_g}$ into $\Sigma_{H_f}$. Analogously, $d\psi_1(., 0)$ maps $\Sigma_{H_f}$ into $\Sigma_{H_g}$. Therefore, $\Sigma_{H_f}$ and $\Sigma_{H_g}$ are bi-Lipschitz equivalent.
\end{proof}

\begin{example}\rm Consider the families $J_{10}$ and $T_{2, 4,5}$ as in Arnold's list (see Section \ref{sec3}). Recall that germs in $J_{10}$ have the form $x^3 + txy^4 + y^6 + Q$ and germs in $T_{2,4,5}$ have the form $x^4 +y^5 +tx^2y^2 + Q$ where $Q$ is a quadratic form in further variables. Put $f(x,y) = x^3 + txy^4 + y^6$ and $g(x,y)= x^4 +y^5 +tx^2y^2$. It is obvious that $m_f \neq m_g$. By Theorem \ref{thm_splitting_weak}, germs in $J_{10}$ and germs in $T_{2,4,5}$ have different bi-Lipschitz types. 
\end{example}

\section{Lipschitz modality of $J_{10}$}\label{section7}
Henry and Parusi\'nski \cite{hp1} proved the the non-quadratic part of $J_{10}$ which is $x^3 + txy^4 + y^6$ is Lipschitz modal. However, this does not imply that $J_{10}$ is Lipschitz modal. Following the idea of Henry--Parusi\'nski  \cite{hp1}, \cite{hp2} we prove in this section that  $J_{10}$ is also Lipschitz modal. 

We fix $\bb K = \bb C$ or $\bb R$ and consider two  bi-Lipschitz equivalent smooth function germs $f, g: \bb K^n, 0 \to \bb K, 0$. Assume that $\varphi: \bb K^n, 0 \to \bb K^n, 0$ is a germ of a bi-Lipschitz homeomorphism such that $f \circ \varphi = g$. Let $L$ be the bi-Lipschitz constant of $h$. It is shown in Lemma \ref{lem_invariant_derivative} that
\begin{equation}\label{equ_J10_1}
L^{-1}\|Df(\varphi(x))\| \leq \|D g(x))\| \leq L \|D f(\varphi(x))\|.
\end{equation}

Given $\delta> 1$, we define
$$V^\delta (f) =\{x : f(x)\neq 0, \delta^{-1} \|x\|\|D f(x)\| \leq \|f(x)\|\leq  \delta \|x\|\|D f(x)\|\}.$$
It follows from (\ref{equ_J10_1}) that
\begin{equation}\label{equ_J_10_2}
V^{L^{-2} \delta} (f) \subset \varphi(V^\delta(g)) \subset V^{L^{2}\delta}(f).
\end{equation}

Given $\sigma >0$, $d >0$, we define 
$$ W^{\sigma}(f, d) = \{x \in \bb K^n: |f(x)| \leq \sigma \|x\|^d\}.$$
It is obvious that
\begin{equation}\label{equ_J_10_3}
W^{\sigma L^{-d}}(f, d) \subset \varphi(W^\sigma (g,d)) \subset W^{\sigma L^d} (f, d).
\end{equation}
Denote by $\Omega_{\sigma,d}^\delta (f) = V^{\delta}(f) \cap W^{\sigma}(f, d).$ It follows from (\ref{equ_J_10_2}) and (\ref{equ_J_10_3}) that
\begin{equation}\label{equ_J10_4}
 \Omega_{\sigma L^{-d},d}^{\delta L^{-2}} (f) \subset \varphi(\Omega_{\sigma,d}^\delta (g))\subset \Omega_{\sigma L^{d},d}^{\delta L^{2}} (f).
\end{equation}

Now we consider $J_{10}$ that consists of germs of the following form 
$$f_\lambda( u, v, w) =   u^3 -3\lambda^2uv^4 +v^6 + w^2$$
 where $(u, v, w) \in \bb K^2 \times \bb K^p$, $w = (w_1, \ldots, w_p)$ and $w^2 := w_1^2 +\ldots + w_p^2$.
 
Fix $\xi = f_{\lambda}$ with $1\pm 2\lambda^3 \neq 0$. Given $\delta, \sigma>0$. Notice that  $\Omega_{\sigma, 6}^\delta (\xi)$ is a semialgebraic set. Suppose that the germ at $0$ of $\Omega_{\sigma, 6}^\delta (\xi)$  has dimension $\geq 1$. Then, we have the following results.

\begin{lem}\label{lem_4.3}  The tangent cone at $0$ of $\Omega^{\delta}_{\sigma,6}(\xi)$ is contained in the $v$-axis.
\end{lem} 

\begin{proof}  Let $\mu = (a, b, c) \in \bb K^{2+p}$ be a unit vector of the tangent cone of $\Omega^{\delta}_{\sigma,6}(\xi)$ at $0$. There exists a real analytic curve $\gamma:[0,\varepsilon) \to \Omega^{\delta}_{\sigma,6}(\xi)$ such that $\gamma(0)=0$ and $\mu = \lim_{t\to 0}\frac{\gamma(t)}{\|\gamma(t)\|}$. Write $\gamma(t) = (\gamma_u(t), \gamma_v(t), \gamma_w(t))$. Reparametrizing $\gamma$ if necessary, we may assume that $\|\gamma(t)\| \sim |t|$.  It suffices to show that $a = c = 0$. 

Indeed, if $c \neq 0$, we have $|\gamma_u(t)| \sim  |t|^\alpha$, $|\gamma_v(t)| \sim  |t|^\beta$ and $|\gamma_w(t)| \sim  |t|$ where $\alpha, \beta \geq 1$ (put $|t|^\infty = 0$ as a convention). This implies $|\xi(\gamma(t))| \sim  |t|^2$. Hence $\gamma(t) \not\subset W^{\sigma}(\xi, 6)$.

If $c = 0, a\neq 0$, we have $|\gamma_u(t)| \sim  |t|$,  $|\gamma_v(t)|\sim |t|^\alpha$  and $|\gamma_w(t)|\sim |t|^\beta$ where $\alpha \geq 1$ and $\beta > 1$.  Recall that $D \xi(u, v, w)= ( 3u^2 -3\lambda^2 v^4, -12\lambda^2 uv^3 + 6v^5, 2w)$. There are two possibilities
\begin{itemize}[leftmargin=.7cm]
\item if $2\beta \leq 3$ then $g(\gamma(t)) \lesssim |t|^{2\beta}$ and $|D \xi(\gamma(t))| \sim |t|^\beta$. Thus, 
$$\|\gamma(t)\|\|D \xi (\gamma(t))\| \sim |t|^{\beta +1} \gg  |t|^{2\beta} \sim |\xi(\gamma(t))|$$ (since $\beta >1$). This implies $\gamma \not\subset V^{\delta}(\xi)$.\\

\item if $2\beta > 3$ then $|\xi(\gamma(t))| \sim |t|^{3}$. Hence, $\gamma(t) \not \subset W^{\sigma}(\xi, 6)$.
\end{itemize}
All the cases above give contradiction, therefore $a$ and $c$ must be equal to $0$.
\end{proof}

\begin{lem}\label{lem_4.4}
On the set $\Omega^{\delta}_{\sigma,6}(\xi)$, we have

(1) $u = \pm v^2 + O(v^3)$,

(2) $g(u,v,w) = (1\mp 2\lambda^3) v^6 + O(v^7)$.
\end{lem}

\begin{proof}
We have $\partial \xi /\partial u = 3(u^2 -\lambda^2v^4)$, $\partial \xi /\partial v = 6v^3(v^2 -2\lambda^2u)$, $\partial \xi /\partial w = 2w$. Let $x= (u, v, w) \in \Omega^{\delta}_{\sigma,6}(\xi)$, then  $\|D \xi (x)\| \|x\| \sim |\xi (x)|$ and
 $|\xi(x)| \lesssim \|x\|^6$. By Lemma \ref{lem_4.3}, we have $\|x\| \sim |v|$. This implies that $\|D \xi(x)\| \lesssim  \|x\|^5 \sim |v|^5$. Therefore,
$$
\begin{cases}
|3(u^2 -\lambda^2v^4)| \lesssim |v|^5,\\
|6v^3(v^2 -2\lambda^2u)| \lesssim |v|^5 \\
2|w|  \lesssim |v|^5
\end{cases}
$$
This implies that $u = \pm \lambda v^2 + O(v^3)$, $w = O(v^5)$, and hence we have formula (1). The second formula then follows directly from (1).  
\end{proof}

\begin{lem}\label{lem_4.5} Let $\{a_k\}, \{b_k\}$ be two sequences of points in $\Omega^{\delta}_{\sigma,6}(\xi)$ converging to $0$ such that $\|a_k - b_k\|\lesssim \|a_k\|^{1+\varepsilon}$ for some $\varepsilon >0$. Then,
$$\lim_{k\to \infty}  \frac{\xi(a_k)}{\xi(b_k)} = c,$$
where $c$ is one of the following values $\{1, \frac{1+2\lambda^3}{1-2\lambda^3}, \frac{1-2\lambda^3}{1+2\lambda^3}\}$.
\end{lem}
\begin{proof} Let $a_k = (u_k, v_k, w_k)$ and $b = (u'_k, v_k', w'_k)$. By Lemma \ref{lem_4.4}, 
$$ g(a_k) = (1\mp2 \lambda^3) v_k^6 + O(v_k^7)$$
and
$$  g(b_k) = (1\mp2\lambda^3) {v'_k}^6 + O({v'_k}^7).$$
Since $a_k \in \Omega_{\sigma, 6}^\delta (\xi)$, by Lemma \ref{lem_4.3}, $|a_k|\sim |v_k|$. Moreover, 
$$|v_k-v'_k| \leq |a_k - b_k| \lesssim |a_k|^{1+\varepsilon} \sim |v_k|^{1+\varepsilon}.$$
The result then follows.
\end{proof}

We now prove the main result of the section.
\begin{thm}\label{prop_j10_modal} Assume that $1\pm 2\lambda^3 \neq 0$ and $1\pm 2\lambda'^3 \neq 0$. If $f_\lambda$ and $f_{\lambda'}$ are bi-Lipschitz equivalent then $\lambda^3 = \pm\lambda'^3$. 
\end{thm}

\begin{proof}
For $1\pm 2\lambda^3 \neq 0$,  we consider the polar curve of $f_\lambda$
$$\Gamma_\lambda = \{f_\lambda/ \partial u = \partial f_\lambda/ \partial w = 0\} = \{u =\pm \lambda v^2, w = 0\}.$$ 
The restrictions of $f_\lambda$ and $D f_\lambda$ to $\Gamma_\lambda$ are
$$f_\lambda|_{\Gamma_\lambda} = (1\pm 2\lambda^3)v^6$$
and 
$$Df_\lambda|_{\Gamma_\lambda} = (0, \mp 12\lambda^3v^5 + 6v^5, 0).$$
We observe that they  satisfy the following conditions: (i) $|f_\lambda(x)| \lesssim \|x\|^6$; (ii) $\|x\|\|Df(x)\| \sim |f_\lambda(x)|$.
Therefore, for a given $L \geq 1$ we can choose $\delta, \sigma >0$ big enough such that the set
 $$\Omega^{\delta L^{-2}}_{\sigma L^{-d},6}(f_\lambda) = V^{\delta L^{-2}} (f_\lambda) \cap W^{\sigma L^{-6}}(f_\lambda,6)$$ contains the polar curve $\Gamma_\lambda$.
 
Now, since $f_\lambda$ and $f_{\lambda'}$ are bi-Lipschitz equivalent, there is a  germ of a bi-Lipschitz homeomorphism $h: (\bb K^{2+p}, 0) \to (\bb K^{2+p},0)$ such that $f_\lambda\circ h = f_{\lambda'}$.  We assume that $L\geq 1$ is the bi-Lipschitz constant of $h$.

Take $d = 6$. By  (\ref{equ_J10_4}), we have 
\begin{equation}\label{equ_J10_5}
 \Omega_{\sigma L^{-6},6}^{\delta L^{-2}} (f_\lambda) \subset h( \Omega_{\sigma,6}^\delta (f_{\lambda'})) \subset \Omega_{\sigma L^{6},6}^{\delta L^{2}} (f_\lambda).
\end{equation}
Choose $\delta, \sigma>0$ large enough such that $\Omega_{\sigma L^{-6},6}^{\delta L^{-2}} (f_\lambda)$ contains the polar curve $\Gamma_{\lambda}$ and $\Omega_{\sigma L^{-6},6}^{\delta L^{-2}} (f_{\lambda'})$ contains the polar curve $\Gamma_{\lambda'}$.

The polar curve $\Gamma_{\lambda} \subset \Omega_{\sigma L^{-6},6}^{\delta L^{-2}} (f_\lambda)$ has two branches $\gamma_1(v) = (\lambda v^2, v, 0)$ and $\gamma_2(v) = (-\lambda v^2, v, 0)$.  The restriction of $f_\lambda$ to these branches are $f_\lambda (\gamma_1(v)) = (1- 2\lambda^3)v^6$ and $f_\lambda( \gamma_2(v)) = (1+ 2\lambda^3)v^6$. Thus, 
$$\lim_{v \to 0} \frac{f_\lambda (\gamma_1(v))}{f_\lambda (\gamma_2(v))} = \frac{1-2\lambda^3}{1+2\lambda^3}.$$
Take $x_k = (\lambda v_k^2, v_k, 0) \in \gamma_1$ and $x'_k = (-\lambda v_k^2, v_k, 0) \in \gamma_2$ where $\lim_{k\to \infty} v_k = 0$. Obviously, 
 $$\lim_{k \to \infty}\frac{f_\lambda(x_k)}{f_\lambda(x_k')} =  \frac{1-2\lambda^3}{1+2\lambda^3}.$$
Put $\tilde{x}_k = h^{-1}(x_k)$ and  $\tilde{x}'_k = h^{-1}(x'_k)$.
By (\ref{equ_J10_5}), $\tilde{x}_k$ and $\tilde{x}'_k $ belong to $\Omega^{\delta}_{\sigma, 6}(f_{\lambda'})$. Since $|x_k - x_k'| \lesssim |x_k|^{2}$ and $h^{-1}$ is bi-Lipschitz, $|\tilde{x}_k - \tilde{x}_k'| \lesssim |\tilde{x}_k|^{2}$. Taking a subsequence if necessary we may assume $\lim_{k\to \infty} \frac{ f_{\lambda'}(\tilde{x}_k )}{ f_{\lambda'}(\tilde{x}'_k)} $ exists, and hence the limit is one of the following values $\{1, \frac{1- 2\lambda'^3}{1+2\lambda'^3} , \frac{1+ 2\lambda'^3}{1-2\lambda'^3} \}$ (see Lemma \ref{lem_4.5}). 

On the other hand,     
$$\lim_{k \to \infty} \frac{f_{\lambda'}(\tilde{x}_k)}{ f_{\lambda'}(\tilde{x}_k')}  = \lim_{k \to \infty}\frac{f_\lambda(x_k)}{f_\lambda(x_k')} =  \frac{1-2\lambda^3}{1+2\lambda^3}.$$
This implies $\frac{1-2\lambda^3}{1+2\lambda^3} \in  \{1, \frac{1- 2\lambda'^3}{1+2\lambda'^3} , \frac{1+ 2\lambda'^3}{1-2\lambda'^3} \}$. The same arguments applied to $\Gamma_{\lambda'}$ we have 
$\frac{1-2\lambda'^3}{1+2\lambda'^3} \in  \{1, \frac{1- 2\lambda^3}{1+2\lambda^3} , \frac{1+ 2\lambda^3}{1-2\lambda^3} \}$. Therefore, $\lambda^3 =\pm \lambda'^3$. \end{proof}

\section{Bi-Lipschitz determinacy of smooth function germs}\label{section8}
This section is devoted to proving that all families of singularities in Theorem \ref{thm_arnold2} which are not enclosed in brackets are bi-Lipschitz trivial. Our idea is to prove the existence of vector fields that satisfy the Thom--Levine criterion (see Theorem \ref{thm_ThomLevine}).
\subsection{A criterion for Lipschitz functions}

We recall a classical criterion to verify if a given function is Lipschitz. 

\begin{lem}\label{lem_bounded_derivative}
Let $U$ be a convex open subset of $\bb R^n$ and $f: U \to \bb R$ be a continuous semialgebraic function. If $Df(x)$ exists and $\|Df(x)\| \leq M$ except finitely many points in $U$, then $f$ is Lipschitz with the Lipschitz constant also bounded by $M$.  
\end{lem}

\begin{proof}
We first prove the statement for $n = 1$. Without loss of generality we can assume that $U =(a,b)$ is an open interval. Then, $Df(x)$ is $C^1$ everywhere except finitely many points $a_1< \ldots < a_k$. 

Let $x_1, x_2 \in (a, b)$. Assume that $x_1 \leq a_i < a_{i+1}< \ldots< a_l\leq x_2$.  We have
\begin{align*}
|f(x_1) - f(x_2)|& \leq |\int_{x_1}^{a_i} Df(x) dx  +  \int_{a_{i+1}}^{a_{i+2}} Df(x)  +\ldots \int_{a_l}^{x_2} Df(x)dx|\\
&\leq M |\int_{x_1}^{x_2} dx | \leq M|x_1-x_2|.
\end{align*}
This implies that $f$ is Lipschitz.

For $n > 1$ fix $x , y \in U$. Set $v = \frac{y-x}{\|x -y\|}$. Define $g :U\supset [0,\|x-y\|] \to \bb R$ by
$$g(t) = f\left(x + tv\right).$$
Notice that $g$ is a semialgebraic function and $|Dg(t)|= |D_vf(x+tv)|$ whenever $D_v(f, x+tv$ exists, where $D_vf$ is the directional derivative of $f$ in the direction $v$. Since $\|v\| =1$, $|D_vf (.)| \leq \|Df(.)\|$. From the hypothesis that $\|Df(.)\| \leq M$ is bounded except finitely many points, so is $|Dg(t)|$.  Then by the  arguments as in the case $n=1$, $g$ is a Lipschitz function with the Lipschitz constant bounded by $M$. Therefore,
\begin{align*}
|g(0) - g(\|x - y \|)| & \leq M \|x - y\|\\
|f(x) - f(y)| & \leq M\|x-y\|.
\end{align*} 
This implies that $f$ is Lipschitz.
\end{proof} 

The Thom--Levine criterion for bi-Lipschitz triviality is as follows:
\begin{thm}\label{thm_ThomLevine}
Let $\bb K = \bb C$ or $\bb R$ and $F : (\bb K^n\times \bb K, 0) \to (\bb K,0)$ be a one-parameter deformation of a germ $f : (\bb K^n,0) \to (\bb K,0)$. If there exists a germ of a continuous vector field of the form 
\begin{equation*}
X(x,t) = \frac{\partial}{\partial t} + \sum_{i=1}^n X_i(x,t) \frac{\partial}{\partial x_i}  
\end{equation*}
Lipschitz in $x$, (i.e. there exists a number $C >0$ with
$$ \|X(x_1, t) - X(x_2, t)\|\leq C \|x_1 - x_2||$$ for all $t$), such that $X.F = 0$, then $F$ is a bi-Lipschitz trivial deformation of $f$. 
\end{thm}

\begin{proof}
The existence of the flow of the vector field $X(x, t)$, denoted $\Phi(x, t)$, follows from the classical Picard--Lindelof theorem. Applying the Gronwall lemma, we can show that $\Phi_t: \bb K^n, 0 \to \bb K^n, 0$ is a bi-Lipschitz homeomorphism.  Therefore, this flow induces the bi-Lipschitz triviality  of $F$. 
\end{proof}

It seems to be unknown whether the converse of the above statement holds. For this reason we call the deformation $F$ in Theorem \ref{thm_ThomLevine} strongly bi-Lipschitz trivial. Fernandes and Ruas \cite{fr} gave sufficient conditions for a deformation of a quasihomogeneous polynomial (in the real case) to be strongly bi-Lipschitz trivial. In this section we first establish a similar result for the complex case. 

From now on we always work with complex function germs. 
\subsection{Bi-Lipschitz determinacy} Let $A \subset \bb Q^n_{\geq 0}$ be a finite set of points. For $r \in \bb Q_+$, denote by $rA = \{r\alpha: \alpha \in A\}$. The Newton polyhedron $\Gamma_+(A)$ of $A$ is defined as the convex hull of the set $\bigcup_{\alpha\in  A} (\alpha + \bb R^n_{\geq 0})$. The union of compact faces of $\Gamma_+(A)$, denoted $\Gamma(A)$, is called the \textit{Newton diagram} of $A$. We say that $A$ is a  \emph{Newton set} if the intersection of  $\Gamma_+(A)$ with each coordinate axis is non-empty and $A$ coincides with the set of vertices of $\Gamma_+(A)$.

Given a Newton set $A$, the \textit{control function associated with $A$} is defined as follows
$$ \rho_A(x) = \sum_{ \alpha = (\alpha_1, \ldots, \alpha_n)\in A}|x_1|^{\alpha_1}\ldots |x_n|^{\alpha_n}.$$

We denote by $\bar{x} =(\bar{x}_1, \ldots, \bar{x}_n)$ the complex conjugate of $x$. Consider a polynomial $f$ in the variables $(x, \bar{x})$ given by 
$$f(x, \bar{x}) = \sum_{\nu, \mu} c_{\nu, \mu} x^\nu \bar{x}^\mu$$
where $x^\nu = x_1^{\nu_1}\ldots x_n^{\nu_n}$ for $\nu = (\nu_1, \ldots, \nu_n) \in \bb N^n$ and $\bar{x}^\mu = \bar{x}_1^{\mu_1}\ldots \bar{x}_n^{\mu_n}$ for $\mu = (\mu_1, \ldots, \mu_n) \in \bb N^n$.  Such a polynomial is called a \textit{mixed polynomial of $x$} (see also \cite{Oka}). The support of $f$ is 
$\supp (f) = \{(\nu_1 +\mu_1,\ldots, \nu_n + \mu_n): c_{\nu, \mu} \neq 0\}.$ The Newton polyhedron and the Newton diagram of $\supp(f)$, denoted $\Gamma(f)$ and $\Gamma_+(f)$ respectively, are called the Newton polyhedron and Newton diagram of $f$.

For $w = (w_1,\ldots,w_n) \in \bb Q_+^n$, the filtration (with respect to $w$) of a mixed monomial $M = x^\nu \bar{x}^\mu$ is defined by 
$$ fil_w(M) = \sum_{j = 1}^n w_j (\nu_j + \mu_j).$$
The filtration of a mixed polynomial $f$, denoted $fil_w(f)$ (or $fil(f)$ if $w$ is clear from the context), is the minimum of the filtrations of the mixed  monomials appearing in $f$.

Let $w = (w_1,\ldots,w_n) \in \bb Q_+^n$ and $d \in \bb Q_+$.  A polynomial $h(z)$ is called quasihomogeneous of the type $(w_1, \ldots, w_n; d)$ if  
$$ h(\lambda^{w_1}x_1, \ldots, \lambda^{w_n} x_n) = \lambda^d h(x), \forall \lambda \in \bb C^{*}.$$ 
A polynomial $g = \sum_{\alpha} a_\alpha  x^\alpha$  can always be written as 
$$ g = g_d + g_{d+1} + \ldots$$
where $g_j = \sum_{fil_w (x^\alpha) = j} a_\alpha x^\alpha$. We call $g_d$ the \textit{initial part} (w.r.t. the weight $w$) of $g$.

A mixed polynomial $f$ is called quasihomogeneous of the type $(w_1, \ldots, w_n; d)$ if 
$$ f(\lambda^{w_1}x_1, \ldots, \lambda^{w_n} x_n, \bar{\lambda}^{w_1}\bar{x}_1, \ldots, \bar{\lambda}^{w_n} \bar{x}_n) = \lambda^d f(x, \bar{x}), \forall \lambda \in \bb C^{*}.$$ 

Suppose that $f(x, \bar{x}) = \sum_{\nu, \mu} c_{\nu, \mu} x^\nu \bar{x}^\mu$ and $\supp(f) \subset \Gamma_+(A)$ for some Newton set $A$. We say that $0$ is a \textit{$\Gamma_+(A)$-isolated} point of $f$ if for all compact faces $\sigma$ of $\Gamma_+(A)$,  $ f (x)|_\sigma = 0$ has no solution in $(\bb C^*)^n$ where $f(x)|_\sigma = \sum_{\nu+\mu \in \sigma} c_{\nu, \mu}x^\nu \bar{x}^\mu$ is the restriction of $f$ to the compact face $\sigma$.

From now on we assume that  $A$ is a Newton set.
\begin{lem}\label{lem_control1} Let $f$ be a mixed polynomial with $\supp(f) \subset \Gamma_+(A)$. Then, there exists a constant $C >0$ such that in a neighbourhood of the origin 
$$ \rho_A(x)\geq C|f(x)| .$$
In particular, if $supp(f)$ lies entirely in the interior of $\Gamma_+(A)$ then 
$$\lim_{x\to 0}\frac{|f(x)|}{\rho_A(x)} =0.$$ 
\end{lem}

\begin{proof} For simplicity we may assume $f$ is a mixed monomial $x^\nu \bar{x}^\mu$. Put $u = \nu +\mu$. 
We have $u \in \Gamma_+(A)$ and  $|x^\nu \bar{x}^\mu| = |x^u| =|x_1|^{u_1}\ldots |x_n|^{u_n}$. 

Now assume on the contrary  that the germ at $(0,0)$ of the set
$$ X = \{(x, c) \in \bb C^n \times \bb R_{+} :\rho(x) < c |x^u|\}$$
is non-empty. Notice that $X$ is a semialgebraic germ of dimension $ \geq 1$. Since $(0,0) \in \overline{X}$, there exists a real analytic curve $\gamma(t) = (\gamma_1(t),\ldots, \gamma_n(t), \gamma_{n+1}(t)):[0, \varepsilon) \to X$, with $\gamma(0) = 0$, 
$|\gamma_1(t)| \sim t^{a_1}, \ldots,|\gamma_n(t)| \sim t^{a_n}, |\gamma_{n+1}(t)| \sim t^{b}$
and $$|\rho_A(\tilde{\gamma}(t))| \lesssim t^b |\tilde{\gamma}(t)^u|,$$ where $\tilde{\gamma}(t):=(\gamma_1(t), \ldots, \gamma_n(t))$.
This implies that 
$$\sum_{\alpha \in A} t^{a_1\alpha_1}\ldots t^{\alpha_n\alpha_n} \lesssim t^b t^{a_1u_1}\ldots t^{a_n u_n}.$$
For $t$ small enough, we have 
$$ \sum_{\alpha \in A} t^{\langle a, \alpha\rangle} <  t^{\langle a, u\rangle}.$$
Hence, $\langle a, u\rangle < \inf_{\alpha \in A} \langle a, \alpha\rangle$. Since $\langle a, u\rangle $ is a linear function on $\Gamma_{+}(A)$, it attains a minimum at one of the vertices of $\Gamma_+(A)$. So $\langle a, u\rangle = \inf_{\alpha \in A} \langle a, \alpha\rangle.$ This gives a contradiction.

If $u$ lies in the interior of $\Gamma_{+}(A)$, then there is a $\delta = (\delta_1, \ldots, \delta_n) \in \bb Q_+^n$ such that $u - \delta$ still lies in $\Gamma_{+}(A)$.  This implies, 
$$\frac{|x^u|}{\rho_A(x)} = \frac{|x^{u-\delta}||x^\delta|}{\rho_A(x)} \leq C |x^\delta|$$
which tends to $0$ when $x$ tends to $0$.
\end{proof}

\begin{lem}\label{lem_control2}
Let $f$ be a mixed polynomial such that $0$ is a $\Gamma_+(A)$-isolated point of $f$. Then, there is a $C>0$ such that in a neighbourhood of the origin
$$ |f(x)| \geq C \rho_A(x).$$
\end{lem}
\begin{proof} Suppose that $f(x) = \sum_{\nu, \mu}c_{\nu, \mu} x^\nu \bar{x}^\mu$. Assume on the contrary that there exists a real analytic curve $\gamma: [0, \varepsilon) \to \bb C^n$, $\gamma(0) =0$ such that $\lim_{t\to 0}\frac{|f(\gamma(t))|}{\rho_A(\gamma(t))} =0$. Write $\gamma = (\gamma_1, \ldots, \gamma_n)$ where 
$\gamma_j(t) = (b_j t^{\beta_j} + o(t^{\beta_j})) e^{i\theta_j(t)}$, $j=1, \ldots, n$. Set $b = (b_1, \ldots, b_n)$, $\beta = (\beta_1, \ldots, \beta_n)$ and $\theta = (\theta_1, \ldots, \theta_n)$. 

Consider the linear function $h(\alpha)=\langle\alpha, \beta\rangle$ defined on $\Gamma_+(A)$. Then, $h(\alpha)$ attains its minimum value $m$ at a compact face, say $\sigma$, of $\Gamma_+(A)$. We can write
$$\rho_A(\gamma(t)) = \sum_{\alpha \in \sigma} t^{\langle \alpha, \beta\rangle} + o(t^m),$$
and 
$$f(\gamma(t)) = \sum_{\nu +\mu \in \supp(f) \cap \sigma} c_{\nu, \mu} b^{\nu +\mu} t^{\langle\alpha, \beta\rangle} e^{i\langle\theta, \nu  - \mu \rangle)} + o(t^m).$$ 

Since $\lim_{t\to 0}\frac{|f(\gamma(t))|}{\rho_A(t)} =0$,  $\sum_{\nu +\mu \in supp(f) \cap \sigma} c_{\nu, \mu} b^{\nu +\mu} t^{\langle\alpha, \beta\rangle} e^{i\langle\theta, \nu  - \mu \rangle)} =0$. This implies that 
$\gamma^*(t):=(b_1t^{\beta_1} e^{i\theta_1}, \ldots, b_n t^{\beta_n} e^{i\theta_n})$ is a solution of $f|_\sigma(x)$. Since $b_j >0$ for all $j = 1, \ldots, n$, $\gamma^*(t)$ is contained in $\bb (\bb C^*)^n$. This contradicts the hypothesis that $0$ is a $\Gamma_+(A)$-isolated point of $f$. 
\end{proof}

Denote by $K_A$ the set of $(n-1)$-dimensional compact faces of $\Gamma_+(A)$. Let $\sigma \in K_A$. Then, there is $d \in \bb Q_+$ such that $\sigma$ belongs to some plane $\Delta_\sigma$ given by  $w_{\sigma,1}x_1 + \ldots + w_{\sigma, n}x_n =d$, where $w_{\sigma, i} \in \bb Q_{\geq 0}$.  We call $w_\sigma=(w_{\sigma, 1}, \ldots, w_{\sigma, n})$ the weight corresponding to $\sigma$ with total weight $d$. Notice that $w_{\sigma, i}$ and $d$ can be chosen to be integers. The following result is a consequence of Lemma \ref{lem_control1}.

\begin{lem}\label{lem_control3}
Let $B = rA$, where $r \in \bb Q_+$.  If $\min_{\sigma\in K_A} \{fil_{w_\sigma}(f) \}\geq rd$ then $$|f(x)|\lesssim \rho_{B}(x).$$
In particular, if $\min_{\sigma \in K_A} \{fil_{w_\sigma}\} > rd$ then $$\lim_{x\to 0}\frac{|f(x)|}{\rho_{B}(x)} =0.$$
\end{lem}

Let $(w_1, \ldots, w_n; 2k)$ be integers. We call
$$\rho_k (x) = |x_1|^{2\alpha_1} + \ldots + |x_n|^{2\alpha_n}$$ where $\alpha_i = k/w_i$  \textit{the standard control function} of type $(w_1, \ldots, w_n; 2k)$.

\begin{lem}\label{lem_control_quasi1} Let $f$ be a mixed polynomial and $\{f_t\}_{t \in U}$ be a smooth deformation of $f$. Suppose that $f_t$ is a quasihomogenous mixed polynomial of type $(w_1, \ldots, w_n; 2k)$ for every $t \in U$ and $f_t^{-1}(0)=\{0\}$. If $U$ is compact then there are $0<C_1 < C_2$ independent of $t$ such that 
$$ C_1 \rho_k(x) \leq |f_t(x)|\leq C_2 \rho_k(x).$$
\end{lem}
\begin{proof}
Let $A$ be the set of vertices of  $\Gamma_+(\rho_k)$. It is obvious that $\Gamma_+(A) = \Gamma_+(\rho_k)$ and $\rho_k (x) = \rho_A(x)$.  Since $\supp(f) \subset \Gamma_+(A)$, there is  $C_2(t) >0$ such that $|f_t| \leq C_2(t) \rho_A$ (see Lemma \ref{lem_control1}). Since $f_t^{-1}(0) =0$, $0$ is a $\Gamma_+(A)$-isolated point of $f_t$. Thus, there is exists $C_1(t)>0$ such that $|f_t| \geq C_1(t) \rho_A$ (see Lemma \ref{lem_control2}). By the continuity in $t$ of $f_t$, $C_1(t)$ and $C_2(t)$ can be chosen to be continuous in $t$.  Then, $C_1 = \min_{t\in U} C_1(t)$ and $ C_2 = \max_{t \in U} C_2(t)$ the desired inequality.  
\end{proof}

\begin{lem}\label{lem_control4} Let $\rho_k$ be the standard control function of the type $(w_1, \ldots, w_n; 2k)$ with $w_1 \leq \ldots\leq w_n$.  Let $U \subset \bb C$ be a compact set, and $\{f_t\}_{t \in U}$ be a continuous family of mixed polynomials. If 
$$fil(f_t) \geq 2k+w_n$$ for every $t\in U$,
then $\frac{f_t}{\rho_k}$ is a germ of a Lipschitz function at $0$ with the Lipschitz constant independent of $t$.
\end{lem}
\begin{proof} Set $g_t(x) = \frac{f_t}{\rho_k}, \forall x\neq 0$ and $g_t(0) =0$. Note that $g$ is a mixed rational function. By Lemma \ref{lem_bounded_derivative}, it suffices to show that $g_t$ is continuous and has the partial derivatives bounded everywhere except $0$. Here, by the partial derivatives we mean the derivatives in variables $x_i$ and $\bar{x}_i$. 

Let $A$ be the set of vertices of $\Gamma_+(\rho_k)$. Then,  $\rho_k = \rho_A$ and $\rho_k^2 \sim \rho_{2A}$ (see Lemma \ref{lem_control3}). Since $\rho_k$ is  quasihomogeneous of the type $(w_1, \ldots, w_n; 2k)$, $\Gamma_+(A)$ has only one $(n-1)$-dimensional compact face. The weight this face is $w= (w_1, \ldots, w_n)$ and the total weight is $d=2k$. 

Since $fil(f_t) > 2k$, $\lim_{x \to 0} \frac{f_t}{\rho_k} =0$ (see Lemma \ref{lem_control1}), $g_t$ is a continuous function.

And, 
$$\frac{\partial g_t}{\partial x_i} = \frac{1}{\rho_k^2}(\rho_k\frac{\partial f_t}{\partial x_i} - f_t\frac{\partial \rho_k}{\partial x_i}) \sim  \frac{1}{\rho_{2A}}(\rho_k\frac{\partial f_t}{\partial x_i} - f_t\frac{\partial \rho_k}{\partial x_i}) .$$
It is clear that
$$fil(\rho_k\frac{\partial f_t}{\partial x_i} - f_t\frac{\partial \rho_k}{\partial x_i}) \geq fil(\rho_k) + fil(f_t) - w_n \geq 4k.$$
By Lemma \ref{lem_control3}, $\frac{\partial g_t}{\partial x_i}$ is bounded near the origin By the continuity of $\frac{\partial g_t}{\partial x_i}$ on the compact set $U$, the bound can be chosen such that it does not depend on $t$.  

The boundedness of 
$\frac{\partial g}{\partial \bar{x}_i }$ is proved in the same way. This completes the proof.
\end{proof}

\begin{thm}\label{thm_5.5}
Let $f : \bb C^n, 0 \to \bb C, 0$ be a germ of a quasihomogeneous polynomial of the type $(w_1,\ldots,w_n; d)$ with isolated singularity at $0$. Let $f_t(x)= f(x) + t \theta(x,t)$ be a smooth deformation of $f$ such that the initial part of $f_t$ w.r.t $(w_1, \ldots, w_n)$ has an isolated singularity at $0$ for every $t$. If $fil(\theta) \geq d + \max_{i, j} \{w_i -w_j\}$ then $f_t$ is  strongly bi-Lipschitz trivial over $U$.
\end{thm}
\begin{proof}
Notice that $\frac{\partial f}{\partial x_i}$ is quasihomogeneous of type $(w_1, \ldots, w_n: s_i)$ where $s_i = d-w_i$. Set 
$$ N^*f = \sum_i \bigg|\frac{\partial f}{\partial x_i}\bigg|^{2\alpha_i} = \sum_i \bigg ( \frac{\partial f}{\partial x_i} \overline{\frac{\partial f}{\partial x_i}}\bigg )^{\alpha_i}$$
where $\alpha_i = \frac{k}{s_i}$, $k = lcm(s_i)$. 

Let us denote by $f_t^{int}$ the initial part of $f_t$ w.r.t the weight $w = (w_1, \ldots, w_n)$. 

It is obvious that $N^*f_t^{int}$ is a quasihomogeneous mixed polynomial of the type $(w_1, \ldots, w_n; 2k)$. Since $f^{int}_t$ has isolated singularity at $0$, by Lemma \ref{lem_control_quasi1}, $N^*f_t^{int} \sim \rho_k$. We have $N^*f_t = N^*f_t^{int} + t R(x,t)$ where $fil(R(x, t)) > 2k$. It follows from Lemma \ref{lem_control3} that $R(x,t)/\rho_{k}(x) \to 0$ when $x \to 0$. Thus, $ N^*f_t \sim \rho_{k}$.  

Now, consider the vector field $W = \sum W_i \frac{\partial}{ \partial x_i}$ where $$W_i := \frac{\partial f_t}{\partial t} \bigg (\overline{\frac{\partial f_t}{\partial x_i}}\bigg )^{\alpha_i} \bigg (\frac{\partial f_t}{\partial x_i}\bigg )^{\alpha_i -1}.$$
We have 
$$ \frac{\partial f_t}{ \partial t} N^*f_t = Df_t(W).$$
Without loss of generality, we may assume that $w_1 \leq \ldots \leq w_n$. We have $fil(\frac{\partial f_t}{\partial t}) \geq d + w_n -w_1$ and 
\begin{align*}
fil\bigg (\bigg (\overline{\frac{\partial f_t}{\partial x_i}}\bigg )^{\alpha_i} \bigg (\frac{\partial f_t}{\partial x_i}\bigg )^{\alpha_i -1}\bigg )& = \alpha_i fil\bigg (\bigg (\overline{\frac{\partial f_t}{\partial x_i}}\bigg)\bigg ) + (\alpha_i -1)fil\bigg (\bigg (\frac{\partial f_t}{\partial x_i}\bigg )\bigg )\\
& = (2\alpha_i -1)  fil\bigg (\bigg (\frac{\partial f_t}{\partial x_i}\bigg)\bigg ) \\
& = (2\alpha_i -1)(d-w_i) \\
& = 2k -d +w_i \geq 2k -d +w_1.
\end{align*}
%(\text{ since } fil\bigg (\bigg (\frac{\partial f_t}{\partial x_i}\bigg ) = fil\bigg (\bigg (\overline{\frac{\partial f_t}{\partial x_i}}\bigg)\bigg )) 
Thus, $fil (W_i) \geq (d +w_n-w_1) +(2k -d +w_1) = 2k +w_n$.  By Lemma \ref{lem_control4}, $\frac{W_t}{\rho_{k}(x)}$ is Lipschitz. Since $\rho_k \sim N^*f_t$, $v_t:=\frac{W_t}{N^*f_t}$ is also a Lipschitz vector field. The flow generated by $v$ gives the required bi-Lipschitz triviality (see Theorem \ref{thm_ThomLevine}). 
\end{proof}

\begin{rem}\label{rem_fr_thm} (1) Notice that Theorem 3.3 in \cite{fr} has a gap when applied to homogeneous germs. The statement of the theorem only holds for sufficiently small $t$. Consider for example $f_t(x,y) = x^2 + y^2 + t(xy + y^3)$.  This family satisfies the hypothesis of the theorem but it is not bi-Lipschitz trivial. It is easy to see that $f_2(x,y) = (x+y)^2 + 2y^3$ cannot be bi-Lipschitz equivalent to $f_t$ ($t \neq 2$), for their tangent cones are different. It seems the mistake first occurs in Lemma 3 in \cite{rm} which is used to prove the theorem.

(2) Theorem \ref{thm_5.5} also holds in the real case.

\end{rem}

An application of Theorem \ref{thm_5.5} is that
\begin{cor}\label{cor_trivial} The singularities $X_9$, $Z_{11}$, $W_{12}$, $P_8$, $S_{11}$, $S_{12}$, $U_{12}$  in Arnold's list are bi-Lipschitz trivial. 
\end{cor}

In the following, we show that certain families are bi-Lipschitz trivial. Theorem \ref{thm_5.5} does not apply to these families. 

\begin{lem}\label{lem_t245} $T_{2, 4,5}: f_t(x, y) = x^4 + y^5 +tx^2y^2 (t\neq 0)$ is strongly bi-Lipschitz trivial.
\end{lem}\label{lem_trivial_T245}
\begin{proof} It suffices to prove that for $t_0 \neq 0$ there is a neighbourhood $U$ of $t_0$ such that $f_t$ is bi-Lipschitz trivial along $U$. We will show that there is continuous vector field  $X(x,y,t) =- \frac{\partial}{\partial t} + X_1(x,y,t)\frac{\partial}{\partial x} +  X_2(x,y,t)\frac{\partial}{\partial y}$ which is Lipschitz in $(x,y)$ such that $X.f_t = 0$. The conclusion will then follow from the Thom-Levine criterion. 

We have
$$\frac{\partial f_t}{\partial x} = 4x^3 +2txy^2, \frac{\partial f_t}{\partial y }= 5y^4 +2tx^2y^2, \frac{\partial f_t}{\partial t}= x^2y^2.$$ 
Put $\rho = \left|\frac{\partial f_t}{\partial x}\right|^2 + \left|\frac{\partial f_t}{\partial y}\right|^2 = |4x^3 + 2txy^2|^2 + |5y^4 + 2tx^2y|^2$. It is a mixed polynomial. Denote by $A$ the set of vertices of $\Gamma_+(\rho)$, i.e. $A = \{(6,0),(2,4),(0,8)\}$. 

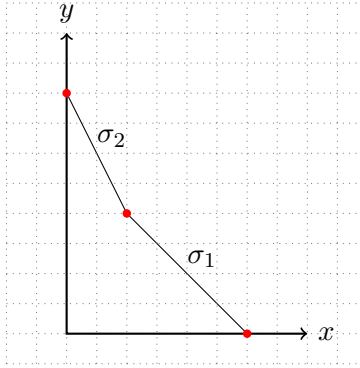
\begin{figure}[h!]
\begin{tikzpicture}[scale=.8]
\draw (3,0) coordinate (a_1) -- (1,2) coordinate (a_2) -- (0,4) coordinate (a_3);
\draw[step=.5cm,gray,dotted] (-1,-.5) grid (5,5.5); 
\draw [<->,thick] (0,5) node (yaxis) [above] {$y$}
|- (4,0) node (xaxis) [right] {$x$};
\fill [red] (a_1) circle (2pt);
\fill [red] (a_3) circle (2pt);
\fill [red] (a_2) circle (2pt);
\node at (2.25,1.25) {$\sigma_1$}; 
\node at (.75,3.25) {$\sigma_2$}; 
\end{tikzpicture}
\caption{Newton diagram of $\Gamma_+{\rho}$}\label{Newtdiagram}
\end{figure}
	
The control function associated with $A$ is $\rho_A = |x|^6 + |x|^2|y|^4 + |y|^8$. Since $supp(A) \subset \Gamma_{+}A$, $\rho \lesssim \rho_A$.  It is easy to check that $0$ is a $\Gamma_{+}(A)$-isolated point of $\rho$, hence by Lemma \ref{lem_control2}, $\rho_A \lesssim \rho$. Thus, $\rho \sim \rho_A$. 

Notice that 
\begin{align}
\label{eq_lem_2}
\rho . \frac{\partial f_t}{\partial t} = \left(\frac{\partial f_t}{\partial t}\overline{ \frac{\partial f_t}{\partial x}}\right)\frac{\partial f_t}{\partial x}  + \left(\frac{\partial f_t}{\partial t}\overline{ \frac{\partial f_t}{\partial y}}\right)\frac{\partial f_t}{\partial y}.
\end{align}
This implies,
\begin{align*}
\frac{\partial f_t}{\partial t} &=\frac{ \left(\frac{\partial f_t}{\partial t}\overline{ \frac{\partial f_t}{\partial x}}\right)}{\rho}\frac{\partial f_t}{\partial x}  + \frac{\left(\frac{\partial f_t}{\partial t}\overline{ \frac{\partial f_t}{\partial y}}\right)}{\rho}\frac{\partial f_t}{\partial y}\\
\end{align*}
Put $a = \frac{\partial f_t}{\partial t}\overline{ \frac{\partial f_t}{\partial x}}$ and $b = \frac{\partial f_t}{\partial t}\overline{ \frac{\partial f_t}{\partial y}}$. Then, set
$$\tilde{a}(x,y,t) = \begin{cases}
\frac{a}{\rho},& (x,y) \neq (0,0)\\
0 & (x,y) = (0,0) 
\end{cases}$$ 
and
$$\tilde{b}(x,y,t) = \begin{cases}
\frac{b}{\rho}, & (x,y) \neq (0,0)\\
0 & (x,y) = (0,0) 
\end{cases}.$$
By the equation (\ref{eq_lem_2}), the vector field  
$$X(x, y, t) = -1\frac{\partial}{\partial t}+\tilde{a}\frac{\partial}{\partial x}+\tilde{b}\frac{\partial}{\partial y}$$
 satisfies   $X.f_t = 0$. We will now show that $X$ is continuous. The Newton diagram of $A$ has two $1$-dimensional compact faces, denoted $\sigma_1$ and $\sigma_2$ (see the Figure \ref{Newtdiagram}).  The weights corresponding to $\sigma_1$ and $\sigma_2$ are $w_{\sigma_1} = (4,4)$ and $w_{\sigma_2} = (6,3)$ with the total weight $d =24$. Computation gives $\min_{i=1,2} \{ fil_{w_{\sigma_i}} (a)\} =28>d $ and $\min_{i=1,2} \{ fil_{w_{\sigma_i}} (b)\} =30>d$. By Lemma \ref{lem_control3}, $\frac{a}{\rho_A}$ and $\frac{b}{\rho_a}$ tend to $0$ when $(x, y)$ tends to $(0,0)$. Since $\rho \sim \rho_A$, $\tilde{a}$ and $\tilde{b}$ are continuous. 

In order to prove that $\tilde{a}$ and $\tilde{b}$ are Lipschitz in $(x, y)$, it suffices to show that  the derivatives with respect to $x, \overline{x}, y$ and $\overline{y}$ are bounded everywhere except $0$. We will show that $\frac{\partial A}{\partial x}$ is bounded, the boundedness of the other partial derivatives can be shown similarly. 

For $(x, y) \neq (0,0)$,
$$
\frac{\partial \tilde{a}}{\partial x} =
\frac{1}{\rho^2}(\rho \frac{\partial a}{ \partial x} - a \frac{\partial \rho}{\partial x}) \sim \frac{1}{\rho_{2A}}(\rho \frac{\partial a}{ \partial x} - a \frac{\partial \rho}{\partial x}) 
$$
Notice that 
$fil(\rho \frac{\partial a}{ \partial x} - a \frac{\partial \rho}{\partial x}) \geq fil(a) + fil(\rho) - w_x$. Restricting to the compact faces $\sigma_1$ and $\sigma_2$, we get
$$fil_{w_{\sigma_1}}((\rho \frac{\partial a}{ \partial x} - a \frac{\partial \rho}{\partial x}) \geq  fil_ {w_{\sigma_1}}(a) + fil_{w_{\sigma_1}} (\rho) - w_{1,x}= 28+24-4 = 48,$$
$$fil_{w_{\sigma_2}}((\rho \frac{\partial a}{ \partial x} - a \frac{\partial \rho}{\partial x}) \geq  fil_ {w_{\sigma_2}}(a) + fil_{w_{\sigma_2}} (\rho) - w_{1,x}= 30 + 24-6  = 48.$$
Thus, $\min_{i =1,2} fil_{w_{\sigma_i}}((\rho \frac{\partial a}{ \partial x} - a \frac{\partial \rho}{\partial x})) = 48 \geq 2d$. By Lemma \ref{lem_control3}, $\frac{\partial A}{\partial x}$ is bounded.
\end{proof}
\begin{rem}\label{rem_triviality}(i) Using the method of Lemma \ref{lem_t245}, one can easily prove that the following families are strongly bi-Lipschitz trivial :

(1) $T_{2,5,5}$, take $\rho = \left|\frac{\partial f_t}{\partial x}\right|^2 + \left|\frac{\partial f_t}{\partial y}\right|^2$,

(2) $T_{p,q,r}$, $3 \leq p, q, r \leq 5$,  take $$\rho  = \left|\frac{\partial f_t}{\partial x}\right|^2 + \left|\frac{\partial f_t}{\partial y}\right|^2 + \left|\frac{\partial f_t}{\partial z}\right|^2,$$

(3) $Q_{10}$, take $$ \rho = \left|\frac{\partial f_t}{\partial x}\right|^4 + \left|\frac{\partial f_t}{\partial y}\right|^4 + |y|^2\left|\frac{\partial f_t}{\partial z}\right|^4.$$
\end{rem}

Notice that all the control functions we have used so far are of the form $\rho = \sum_{i =1}^n a_i \left|\frac{\partial f_t}{\partial x_i}\right|^2$. It is easy to see that $\rho \in Jf$ since we can write  
$$\rho = \sum_{i =1}^n a_i \overline{\frac{\partial f_t}{\partial x_i}} \frac{\partial f_t}{\partial x_i}.$$
This implies that
$$\frac{\partial f_t}{\partial t} = \sum_{i =1}^n  \frac{ a_i \frac{\partial f_t}{\partial t}\overline{\frac{\partial f_t}{\partial x_i}}}{\rho} \frac{\partial f_t}{\partial x_i}.$$
If the vector field  $W = (W_1, \ldots, W_n)$ where $W_i =  \frac{ a_i \frac{\partial f_t}{\partial t}\overline{\frac{\partial f_t}{\partial x_i}}}{\rho}$ is Lipschitz  then it satisfies the Thom--Levine criterion. Checking the Lipschitzness of $W$ is quite easy because we know exactly what the $W_i$'s are. However, the vector field $W$ induced by the control function $\rho$ of the form above is not always Lipschitz. For example, in the case $Q_{11}$, we could not find a control function of the same type.  For this reason, in the following we suggest another control function for $Q_{11}$ that belongs to $Jf$. This control function  does not give precisely what the vector field $W$ is, but it  enough for us to prove the bi-Lipschitz triviality of $Q_{11}$.

\begin{lem}\label{lem_q11trivial}
	$f(x,y,z) = x^3 + y^2z + xz^3 + t z^5$ is strongly bi-Lipschitz trivial.
\end{lem}

\begin{proof}
	The partial derivatives of $f$ are $\frac{\partial f}{\partial x} = 3x^2 + z^3$, $\frac{\partial f}{\partial y} = 2yz$, $\frac{\partial f}{\partial z} = y^2 + 3xz^2 +5tz^4$ and $\frac{\partial f}{\partial t} = z^5$. Set $\rho = |x|^{16} + |y|^{16} + |z|^{16}$. It is obvious that $\rho$ is a standard control function of type $(1, 1, 1; 16)$. We will show that there exist germs of mixed polynomials $\varepsilon_i(x, y,z,t)$ ($i = 1,2,3$) satisfying
	\begin{align}\label{eq_Q11_1}
	\rho = \varepsilon_1 \frac{\partial f}{\partial y} + \varepsilon_2 \frac{\partial f}{\partial y} + \varepsilon_3 \frac{\partial f}{\partial z},
	\end{align}
such that $\frac{{\varepsilon_i z^5}}{\rho}$ is a Lipschitz function germ in the variables $(x, y,z)$ for each $i = 1,2,3$. Multiplying the equation (\ref{eq_Q11_1}) by $\frac{\partial f}{\partial t} = z^5$ we get,
	\begin{align}\label{eq_Q11_2}
	\frac{\partial f}{\partial t} =  \left(\frac{\varepsilon_1 z^5}{\rho}\right) \frac{\partial f}{\partial x} + \left( \frac{\varepsilon_2 z^5} {\rho}\right ) \frac{\partial f}{\partial y} + \left(\frac{\varepsilon_3 z^5}{\rho}\right) \frac{\partial f}{\partial z}.
	\end{align}
The lemma will then follow from the Thom-Levine criterion.	

Observe that, for $\frac{\varepsilon_i z^5}{\rho}$ to be a Lipschitz function germ, it is enough to show that its partial derivatives in the variables $x, \bar{x}, y, \bar{y}, z$ and $\bar{z}$ are bounded. Set $\varepsilon_i' = \varepsilon z^5$.  We have	
	$$\frac{\partial}{\partial u} \bigg(\frac{\varepsilon_i z^5}{\rho} \bigg) =  \frac{\frac{\partial \varepsilon'_i}{\partial u}\rho -\varepsilon'_i \frac{ \partial \rho}{\partial u}}{\rho^2},$$
where $u \in \{x, \bar{x},y,\bar{y}, z, \bar{z}\}$.

Put $N = \frac{\partial \varepsilon'_i}{\partial u}\rho -\varepsilon'_i \frac{ \partial \rho}{\partial u}$. Notice then that $$fil (N) \geq fil(\rho) + fil(\varepsilon'_i) - w_u,$$ where $w_u$ is the the weight of $u$. By Lemma \ref{lem_control1}, if 
\begin{equation}\label{eq_Q11_3}
fil(\rho) + fil(\varepsilon'_i) - w_u \geq fil(\rho^2) = 2fil(\rho),\,\,\, \forall u \in \{x, \bar{x}, y,\bar{y}, z, \bar{z}\}
\end{equation}
then $\frac{\partial}{\partial u} \bigg(\frac{\varepsilon_i z^5}{\rho} \bigg)$ is bounded. Since the weights of $x, \bar{x}, y, \bar{y}, z, \bar{z}$ are equal to $1$, the inequality (\ref{eq_Q11_3}) is equivalent to 
$fil(\varepsilon'_i)  \geq 17$, or $fil (\varepsilon_i) \geq 12$.

Since $\rho = x^8 \bar{x}^8 + y^8 \bar{y}^8 + z^8 \bar{z}^8$, to prove (\ref{eq_Q11_1}) it suffices to prove that $x^8$, $y^8$ and $z^8$ belong to $Jf$ and their coefficients have filtrations $\geq 4$. Here, by coefficients of a polynomial $h = h_1\frac{\partial f}{\partial x} + h_2 \frac{\partial f}{\partial y} + h_3\frac{\partial f}{\partial z}\in Jf$ we mean $(h_1, h_2, h_3)$. 

Observe that,
\[
	\begin{array}{rcccccccccc}
	\frac{\partial f}{\partial x}\cdot x^6    & = & 3x^8 &+& x^6z^3  & \\       
	\frac{\partial f}{\partial x}\cdot x^4z^3 & = &      & & 3x^6z^3 &+& x^4z^6      
	
	\end{array}
	\]
This gives
\begin{equation}\label{eq_Q11_4}
\left({\begin{array}{cc}
 3 & 1 \\
 0 & 3 \\
 \end{array}}\right)
\left(\begin{array}{c}
 x^8\\
 x^6z^3\\
 \end{array}\right)
 =
\left(\begin{array}{c} 
\frac{\partial f}{\partial x}\cdot x^6 \\ 
\frac{\partial f}{\partial x}\cdot x^4z^3 -x^4z^6\\
\end{array}\right)
\end{equation}
It is easy to check that $z^6 \in Jf$ (see the Gr\" oebner basis of $Jf$). Then, $x^4z^6 \in Jf$ and its coefficients have filtrations $\geq 4$. Since the system (\ref{eq_Q11_4}) has unique solution, $x^8 \in Jf$. The filtrations of $x^6$ and $x^4z^3$ are $\geq 4$, this implies that the coefficients of $x^8$ have filtrations $\geq 4$.

Similarly,
	\[
	\begin{array}{rcccccccccc}
	\frac{\partial f}{\partial x}\cdot z^5 &=& z^8 &+& 3x^2 z^5 & &        & &         & & \\
	\frac{\partial f}{\partial z}\cdot xz^3&=&     & &  3x^2z^5 &+& 5txz^7 &+& xz^3y^2 & &\\
	\frac{\partial f}{\partial z}\cdot z^5 &=&     & &          & & 3xz^7  &+& 5tz^9   &+& z^5y^2\\
	\end{array}
	\]  
This gives
\begin{equation}\label{eq_Q11_5}
\left({\begin{array}{ccc}
 1 & 3 & 0 \\
 0 & 3 & 5t\\
 0 & 0 & 3
 \end{array}}\right)
\left(\begin{array}{c}
 z^8\\
 x^2z^5\\
 xz^7 
 \end{array}\right)
 =
\left(\begin{array}{c} 
\frac{\partial f}{\partial x}\cdot z^5 \\ 
\frac{\partial f}{\partial z}\cdot xz^3 -xz^3y^2\\
\frac{\partial f}{\partial z}\cdot z^5 - 5tz^9 - z^5 y^2
\end{array}\right)
\end{equation}
Put $g(x, z) = f(x, 0, z)$. We have $\frac{\partial g}{\partial x } = \frac{\partial f}{\partial x}$ and $\frac{\partial g}{\partial z} = \frac{\partial f}{\partial z} - y^2$. Checking the Gr\" oebner basis of $Jg$, we have $z^5 \in Jg$. Therefore, there exist polynomials $\eta_1$ and $\eta_2$ such that 
$$
	z^5 = \eta_1 \frac{\partial g}{\partial x} + \eta_2 \frac{\partial g}{\partial z}= \eta_1 \frac{\partial f}{\partial x} + \eta_2 (\frac{\partial f}{\partial z} - y^2)
$$
Thus,
	\begin{align*}
	z^9 &= \eta_1\cdot z^4 . \frac{\partial f}{\partial x} + \eta_2 \cdot z^4 \cdot  \frac{\partial f}{\partial z} - \eta_2 \cdot z^4 y^2\\
	&= \eta_1\cdot y^4 . \frac{\partial f}{\partial x} + \eta_2 \cdot y^4 \cdot\frac{\partial f}{\partial z} - \frac{1}{2}\eta_2\cdot z^3y\cdot \frac{\partial f}{\partial y}. 
	\end{align*}
This shows that $z^9 \in Jf$ and its coefficients have filtrations $\geq 4$. Also, $xz^3y^2 = \frac{1}{2} xz^2y \cdot\frac{\partial f }{\partial y}$, $z^5y^2 = \frac{1}{2} z^4y\cdot \frac{\partial f }{\partial y}$, therefore, they belong to $Jf$ and the filtrations of their coefficients $\geq 4$. The system (\ref{eq_Q11_5}) has unique solution, this implies that $z^8 \in Jf$ with coefficients of filtrations $\geq 4$. 
	
Finally $y^8 = y^6 \cdot\frac{\partial f}{\partial z} -\frac{1}{2}(3xy^5z + 5ty^5z^3)\cdot\frac{\partial f}{\partial x}$, $y^8 \in Jf$ and its coefficients have filtrations $\geq 6$.  
\end{proof}

\section{Classification of Lipschitz simple germs}\label{sec9}
\subsection{Adjacencies}
Consider the list of adjacencies given in Theorem \ref{thm_arnold2}.  We know from a result of Brieskorn \cite{Brieskorn} that the singularities not appearing in the brackets do not deform to $J_{10}$. In this section we prove that the singularities enclosed in the brackets deform to $J_{10}$. 

Let $f_t(x,y) = x^3 + txy^4 + y^6 \in \ak m_2$. Give $(x,y)$ the weight $w=(2,1)$. Then, $fil_w(f_t) = 6$. It follows from Arnold's result \cite{Arnold1} that 
\begin{lem}\label{lem_arnold_j10}
(i) Any polynomial $g (x, y) \in \ak m_2$ with isolated singularity at $0$ and $fil_w(g) =6$ is smoothly equivalent to $f_t$ for some $t$;

(ii) If $g(x, y)$ is a monomial such that $fil_w(g) \geq 7$ then $f_t + g$ is smoothly equivalent to $f_t$.
\end{lem}

We prove further that,

\begin{lem}\label{lem_deform_j10} Let $f \in \ak m_n^2$ be a polynomial with isolated singularity at $0$. Let $w_1 =2, w_2 =1, w_j = 3, \forall j \geq 3$, and let $w = (w_1, \ldots, w_n)$. If $fil_w (f) \geq 6$, then $f$ deforms to $J_{10}$. 
\end{lem}
\begin{proof}
First, we take the deformation as follows
$$F = f + \sum_{j\geq 3} a_j x_j^2.$$
By the Splitting lemma, we know that $F$ is smooth equivalent to $g(x_1, x_2) + \sum_{j\geq 3}\tilde{a}_j x_j^2$. Notice that $fil_w (g) \geq  fil_w (f)$ (see  Gibson \cite{Gibson}, page 126, the algorithm in the proof of the Splitting lemma). It suffices to prove that $g(x_1, x_2) + \sum_{j\geq 3}\tilde{a}_j x_j^2$ deforms to $J_{10}$. Indeed, consider the following deformation:
$$G = g(x_1, x_2) + \sum_{j\geq 3} \tilde{a}_j x_j^2 + b_1 x_1^3 + b_2 x_2^6 + b_3 x_1x_2^4.$$
It follows from Lemma \ref{lem_arnold_j10} that $g(x_1, x_2)  + b_1 x_1^3 + b_2 x_2^6 + b_3 x_1x_2^4$ is smoothly equivalent to $f_t(x_1, x_2)$ for some $t$. This implies that $G$ is $J_{10}$. 
\end{proof}

\begin{thm}\label{thm_deformto} All germs belonging to the families enclosed in the brackets in the diagram in Theorem \ref{thm_arnold2} deform to $J_{10}$. 
\end{thm}
\begin{proof} 
(1) That $T_{p, q,6}, p \geq 2, q \geq 3$, $Z_{12}$ and $W_{13}$ deform to $J_{10}$ is a direct consequence of Lemma \ref{lem_deform_j10} by giving $(x, y, z)$ the weight $w= (2,1,3)$.

(2) $Q_{12}\rightarrow J_{10}$. We have $$ Q_{12}:  x^3 + y^5 + yz^2 + a xy^4.$$ 
By the change of coordinates of type $z\mapsto z + i y^2$ and leaving other coordinates fixed, we can eliminate $y^5$. That is,  $Q_{12}$ is smooth equivalent to 
$$g(x, y, z)= x^3 + yz^2 + 2iy^3z + a xy^4.$$
Let $w = (2,1,3)$ be the weight of $(x,y, z)$. Then, $fil_w (g) =6$. This implies that $g$ deforms to $J_{10}$ (see Lemma \ref{lem_deform_j10}).

(3) $S_{1,0}\rightarrow J_{10}$. Recall that $S_{1,0}$ has the normal form $x^2z+ yz^2 +y^5 +a_1zy^3 +a_2zy^4$. First, by the change of coordinate $z \mapsto z + Ay^2$, we have $S_{1,0}$ is smooth equivalent to
$$g(x,y, z) = Ay^2x^2+zx^2+bAy^6+(A^2+Aa+1)y^5+bzy^4+(2A+a)zy^3+z^2y$$
Choose $A$ such that $A^2+Aa+1=0$ to eliminate $y^5$.  Give $(x,y, z)$ the weight $w = (2,1,3)$. Obviously, $fil_w(g) = 6$. By Lemma \ref{lem_deform_j10}, $g$  deforms to $J_{10}$. 

(4) $U_{12}\rightarrow J_{10}$. Recall that $U_{12}$ has the normal form $x^3 +y^3 +z^4 + axyz^2$. Take the deformation as suggested in \cite{Brieskorn}:
$$x^3 - y^3 +z^4 + axyz^2 + t^2(x+y)^2 + 2t(x+y) z^2.$$
By the change of coordinate $ x\mapsto x-y$, we have
$$x^3 -2y^3 + 3xy^2 + z^4 + axyz^2  - ay^2z^2 + axyz^2 + 2txz^2 +t^2x^2.$$
By changing $x \mapsto x -\frac{1}{2t^2}(x^2 -3xy +3y^2 +2tz^2)$, the term $z^4$ will be eliminated, and the polynomial obtained has filtration with respect to the weight $w = (3,2,1)$ is equal to  $6$, so it deforms to $J_{10}$ (see Lemma \ref{lem_deform_j10}).

(5) $(\bff O) \rightarrow J_{10}$.  Since every germ of corank $\geq 4$ can deform to a germ of corank $4$, it suffices to prove that all corank $4$ germs deform to $J_{10}$. 

Let $g$ be a corank $4$ germ. By the Splitting lemma, we may assume (up to addition of a non-degenerate quadratic form of further variables) that $g \in \ak m_4^3$.  Consider a small deformation of $g$ of the form 
$h_\lambda = g + \lambda (x^3 + y^3 +z^3+w^3)$. Notice that for $\lambda \neq 0$ near $0$, $h_\lambda =  \lambda_1 x^3 + \lambda_2 y^3 + \lambda_3 z^3 + \lambda_4 w^3 +  \tilde{p}$  for some $\lambda_i \neq 0, i =1, \ldots, 4$ and  $\tilde{p}$ is a polynomial of order greater than or equal to $3$ which does not contain the terms $\{x^3, y^3, z^3, w^3\}$. By a change of coordinates we may assume $\lambda_i =1, \forall i = 1,\ldots, 4$. Transversal unfolding shows that $j^3(h_\lambda)$ is smooth equivalent to  
$ f_{a,b,c,d}(x, y, z, w) = x^3 +y^3 +z^3 +w^3 + axyz + bxyw + c xzw + d yzw,$
for some $a, b , c, d$ (see \cite{Gibson}, Chapter III). 

Since $Jf \ak m_4^2 + \ak m_4^5 + \langle xyzw \rangle = \ak m_4^4$, by the complete transversal method (see \cite{bkp}), every germ whose $3$-jet is equivalent to $f_{a, b, c, d}$ has $4$-jet equivalent to $ f_{a, b, c, d, e} :=f_{a, b, c, d} + e xyzw$ for some $e$. It is easy to check that $h_\lambda$ is $4$-determined, thus it is equivalent to $f_{a, b, c, d, e}$ for some $a, b, c, d, e$. 

It remains to show that $f_{a, b, c, d, e}$ deforms to $J_{10}$. First, we take the change of coordinate as follows: $x \mapsto x-z +mw+ \frac{a}{3}y, z \mapsto x+z, w \mapsto w + lx$ for $m, l \in \bb C$. Then, deform the obtained germ with $tx^2 +tw^2$. Next, we  use the algorithm of the Splitting lemma to eliminate the terms of degree $3$ containing either $x$ or $w$ or both. The resulting polynomial, let say $g_{a,b,c,d,e}$, includes monomials of filtration (with respect to the weight  $(3, 2,1,3)$) greater than or equal to $6$ except $z^4$ and $zy^3$. Since the coefficients of $z^4$ and $zy^3$ depend on two parameters $m$ and $l$, there is a closed algebraic set $S \subset \bb C^5$ of dimension less than $5$ such that for all $(a, b, c, d, e) \in \bb C^5 \setminus S$ we can choose $m, l$ such that the coefficients of $z^4$ and $zy^3$ equal to zero.  By Lemma \ref{lem_deform_j10}, $g_{a, b, c, d, e}$ deforms to $J_{10}$  for $(a, b, c, d, e) \in \bb C^5 \setminus S$. Since $\dim S < 5$, thus $g_{a, b, c, d, e}$ deforms to $J_{10}$ for all $a, b, c, d, e$. 
\end{proof}

\subsection{Classification of Lipschitz simple germs}

We are now in position to state and prove our main result. Observe that

(1) smooth simple germs are Lipschitz simple;

(2) $J_{10}$ is Lipschitz modal (see Theorem \ref{prop_j10_modal});

(3) if a germ deforms to a class of Lipschitz modal singularities, then it is also Lipschitz modal;

\begin{thm}\label{thm_corank2}
A germ $f \in \ak m^2_n$ of corank $1$ or $2$ is Lipschitz simple if and only if it is bi-Lipschitz equivalent to one of the germs in the table below:

\begin{center}
\footnotesize
\begin{longtable}{|c|c|c|c|c|}
\caption{Corank $1$ and $2$ Lipschitz simple germs}
\label{table_corank2}\\
\hline
Name & Smooth normal form & Lipschitz normal form &  Codimension & Corank \\ \hline
$A_k$        & $x^{k+1}$     & $x^{k+1}$              & $k \geq 1$ & $1$\\\hline 
$D_k$        & $x^2y +y^{k-1}$  & $x^2y +y^{k-1}$          & $k\geq 4$& \\
$E_6$        & $x^3 + y^4$  & $x^3 + y^4$               &  $6$& \\
$E_7$        & $x^3 + xy^3$     & $x^3 + xy^3$                & $7$& \\
$E_8$        & $x^3 + y^5 $  & $x^3 + y^5 $                & $8$& $2$\\
$X_9$        & $x^4+y^4 + ax^2y^2$              & $x^4+y^4 + x^2y^2$              & $9$& \\
$T_{2,4,5}$    & $x^4 + y^5 + ax^2y^2)$        & $x^4 + y^5 + x^2y^2$        & $10$ & \\
$T_{2,5,5}$	 & $ x^5 +y^5 +  ax^2 y^2 $	      & $ x^5 +y^5 +  x^2 y^2 $	       & $11$ & \\
$Z_{11}$	 & $x^3 y + y^5 +  axy^4$	   & $x^3 y + y^5 +  xy^4$&  $11$ &\\	
$W_{12}$     & $x^4  + y^5 + ax^2y^3$       & $x^4  + y^5 + x^2y^3$        & $12$ & \\
\hline 
\end{longtable}
\end{center}
\end{thm}

\begin{proof}
By Theorem \ref{thm_deformto}, germs of corank $2$ which are not equivalent to germs in the list deform to $J_{10}$ and hence are not Lipschitz simple.

The singularities $A_k, (k\geq 1)$, $D_k, (k \geq 4)$, $E_6$, $E_7$ and $E_8$ are smooth simple germs (see Theorem \ref{thm_arnold1}), they are obviously Lipschitz simple germs. 

Now we prove that the germs in $X_{9}$ are Lipschitz simple. Let $f\in X_{9}$ and $k \in \bb N$ be sufficiently large for $f$. Since corank and codimension of singularities are upper semicontinuous, the only singularities contained in a small enough neighbourhood $U$ of $j^kf(0)$ are $X_9$ or $ADE$-singularities. It is shown in Corollary \ref{cor_trivial} that $X_9$ is a bi-Lipschitz trivial family. We know, moreover, that $ADE$-singularities belong to only finitely many bi-Lipschitz classes. This implies that $f$ is a Lipschitz simple germ. 

The proof for $T_{2,4,5}$ is as follows. Let $g \in T_{2,4,5}$ and $m \in \bb N$ be sufficently large for $g$. Then, again the only singularities contained in the small neighbourhood $V$ of $j^mg(0)$ are those with corank $\leq 2$ and codimension $\leq 10$.  By Brieskorn's result \cite{Brieskorn}, $T_{2,4,5}$ cannot deform to $J_{10}$. Shrinking $V$ if necessary, we can assume that $V$ does not contain germs of $J_{10}$. Thus, singularities in $V$ are either contained in $T_{2, 4,5}$ or $X_9$ or  are $ADE$-singularities. It is proved in Lemma \ref{lem_t245} that $T_{2,4,5}$ is bi-Lipschitz trivial. We also know that all germs in $X_9$ and $ADE$-singularities belong to finitely many bi-Lipschitz classes. Therefore, $g$ is Lipschitz simple. 

The proof for the other germs is similar by induction on the codimension and using the fact that they belong to bi-Lipschitz trivial families (see Corollary \ref{cor_trivial}, Lemma \ref{lem_trivial_T245} and Remark \ref{rem_triviality}) and none of them deform to a class of Lipschitz modal singularities (see Brieskorn's result \cite{Brieskorn}).

It remains to prove that every normal form in the third column of Table \ref{table_corank2} represents a unique bi-Lipschitz class. Notice that except $T_{2,5,5}$ and $Z_{11}$, they all have different codimensions.  The normal forms of $T_{2,5,5}$ and $Z_{11}$ are given by $f = x^5 +y^5 +  x^2 y^2$ and $g = x^3 y + y^5 +  xy^4$. It is easy to check that $\Sigma_{H_f} = \{(x, y) \in \bb C^2: xy = 0\}$ while $\Sigma_{H_g} = \{(x, y)\in \bb C^2: x = 0\}$. Therefore, by Theorem \ref{thm_splitting_weak},  $T_{2,5,5}$ and $Z_{11}$ are of different bi-Lipschitz types.  
\end{proof}

\begin{thm}\label{thm_corank3}
A germ $f \in \ak m^2_n$ of corank $3$ is Lipschitz simple if and only if it is bi-Lipschitz equivalent to one of the germs in the table below:

\begin{center}
\footnotesize
\begin{longtable}{|c|c|c|c|c|}
\caption{Corank $3$ Lipschitz simple germs}
\label{table_corank3}\\
\hline
Name & Smooth normal form  & Lipschitz normal form             & Codimension & Corank \\ \hline
$P_8 = T_{3,3,3}$        & $x^3+y^3+z^3 + axyz$         & $x^3+y^3+z^3 + xyz$         & $8$   & \\
$T_{3,3,4}$  & $x^3 + y^3 + z^4 + a xyz $    & $x^3 + y^3 + z^4 +  xyz $    & $9$  &  \\
$T_{3,3,5}$  & $x^3 +  y^3 + z^5 + a xyz $ & $x^3 +  y^3 + z^5 +  xyz$  & $10$ &  \\
$T_{3,4,5}$  & $x^3 +  y^4 + z^5 +  a xyz  $  &$x^3 +  y^4 + z^5 +  xyz $  & $11$ &  \\
$T_{3,5,5}$  & $x^3 +  y^5 + z^5 +  axyz $  & $x^3 +  y^5 + z^5 +  xyz $  & $12$ &  \\
$T_{4,4,4}$  & $x^4 +  y^4 + z^4 +  axyz $  & $x^4 +  y^4 + z^4 +  xyz $  & $11$ &  \\
$T_{4,4,5}$  & $x^4 +  y^4 + z^5 +  axyz $ & $x^4 +  y^4 + z^5 +  xyz $  & $12$ &  \\
$T_{4,5,5}$  & $x^4 +  y^5 + z^5 +  axyz $ & $x^4 +  y^5 + z^5 +  xyz $  & $13$ &  $3$ \\
$T_{5,5,5}$  & $x^5 +  y^5 + z^5 +  axyz $  & $x^5 +  y^5 + z^5 +  xyz $  & $14$ &  \\
$Q_{10}$     & $x^3 + y^4 + yz^2 + axy^3$     & $x^3 + y^4 + yz^2 + xy^3$  & $10$  & \\
$Q_{11}$     & $x^3 + y^2z + xz^3 + az^5$ & $x^3 + y^2z + xz^3 + z^5$   & $11$  &  \\
$S_{11}$	 & $x^4 + y^2z + xz^2 + a x^3z$ & $x^4 + y^2z + xz^2 + x^3z$ & $11$ &  \\
$S_{12}$	 & $x^2y + y^2z + xz^3 + a z^5$ & $x^2y + y^2z + xz^3 + z^5$  & $12$ & \\
\hline 
\end{longtable}
\end{center}
\end{thm}

\begin{proof}
By Theorem \ref{thm_deformto}, all germs of corank 3 which are not in the above list deform to $J_{10}$, hence they are Lipschitz modal. We know from Corollary \ref{cor_trivial}, Lemma \ref{lem_q11trivial} and Remark \ref{rem_triviality} that every singularity class in the above list is bi-Lipschitz trivial. Moreover, by the result of Brieskorn \cite{Brieskorn} none of the singularities are adjacent to any class of Lipschitz modal singularities so by arguments as in the proof of Theorem \ref{thm_corank2}, they are Lipschitz simple. 

Now let us prove that every normal form in  the third column of Table \ref{table_corank3} represents a unique bi-Lipschitz class. Notice that except $Q_{11}$ and $S_{11}$, all germs have different bi-Lipschitz type, for one of the bi-Lipschitz invariants, namely either the codimension or the singular part of the tangent cone, associated with them are different. We list the invariants for these germs below (see Table \ref{table_corank3_proof}).
\begin{center}
\footnotesize
\begin{longtable}{|c|c|c|c|c|}
\caption{Singular sets of the algebraic tangent cone of normal forms}
\label{table_corank3_proof}\\
\hline
Name & Lipschitz normal form ($f$)  & $\Sigma_{H_f}$ & Codimension \\ \hline
$T_{3,3,5}$  & $x^3 +  y^3 + z^5 +  xyz $ & $y$-axis $\cup$ $z$-axis & $10$   \\
$Q_{10}$     & $x^3 + y^4 + yz^2 + xy^3$     & $y$-axis  & $10$   \\
\hline
$T_{3,4,5}$  & $x^3 +  y^4 + z^5 +  xyz  $  &$y$-axis $\cup$ $z$-axis  & $11$  \\
$T_{4,4,4}$  & $x^4 +  y^4 + z^4 +  xyz $  & $x$-axis $\cup$ $y$-axis $\cup$ $z$-axis & $11$   \\
$Q_{11}$     & $x^3 + y^2z + xz^3 + z^5$ & $y$-axis   & $11$    \\
$S_{11}$	 & $x^4 + y^2z + xz^2 + x^3z$ & $x$-axis & $11$   \\
\hline
$T_{3,5,5}$  & $x^3 +  y^5 + z^5 +  xyz $  & $y$-axis $\cup$ $z$-axis  & $12$   \\
$T_{4,4,5}$  & $x^4 +  y^4 + z^5 +  xyz $ & $x$-axis $\cup$ $y$-axis $\cup$ $z$-axis  & $12$  \\
$S_{12}$	 & $x^2y + y^2z + xz^3 + z^5$ &  $z$-axis  & $12$  \\
\hline 
\end{longtable}
\end{center}
Lemma \ref{lem_Q11_S11} shows that the bi-Lipschitz types of $Q_{11}$ and $S_{11}$ are different.
\end{proof}

\begin{lem}\label{lem_Q11_S11} $Q_{11}$ and $S_{11}$ are not in the same bi-Lipschitz class.
\end{lem}
\begin{proof}
Set $f(x, y, z) =  x^3 + y^2z + xz^3 + z^5$ and $g(x, y, z) = x^4 + y^2z + xz^2 + x^3z$. Recall that the germs in $Q_{11}$ are bi-Lipschitz equivalent to $f + Q$, and germs in $S_{11}$ are bi-Lipschitz equivalent to $g + Q$ where $Q = w_1^2 + \ldots + w_{n-3}^2$. We need to show that $f+Q$ and $g+Q$ are not bi-Lipschitz equivalent. In \cite{Fui}, Theorem 4.2, Bavi\`a-Ausina and Fukui proved that the log canonical threshold of the Jacobian ideal, let us denote by $\lct$, is a bi-Lipschitz invariant. We show that $\lct(J(f+Q))  \neq \lct(J(g+Q))$. By Lemma 3.2 in Totaro \cite{Totaro}, $\lct(J(f+Q)) = \lct(Jf) + \lct (JQ)$ and $\lct(J(g+Q)) = \lct(Jg) + \lct (JQ)$. Thus, it suffices to prove that $\lct (Jf) \neq \lct(Jg)$. To calculate the log canonical threshold of a Jacobian ideal we follow the method suggested in \cite{Fui}  using the Newton polyhedron. Let us start with computing $\lct(Jf)$.

Set $$I = Jf = \langle 3x^2 + z^3, 2yz, y^2 + 3xz^2 +5z^4\rangle .$$
Let $I^0$ denotes the monomial ideal associated with the Newton polyhedron $\Gamma_+(I)$ of $I$ (see the definitions in \cite{Fui}, page 8.1). Then, $I^0 =  \langle x^2, yz, y^2, z^3 \rangle$. By formula (4.1) in \cite{Fui}, we have $\lct(I^0) = \frac{3}{2}$. Consider the map 
$$ \alpha = (\alpha_1, \alpha_2, \alpha_3)= \bigg(\frac{\partial f}{\partial x},  \frac{\partial f}{\partial y}, \frac{\partial f}{\partial z}\bigg).$$
It is obvious that $\alpha_1, \alpha_2, \alpha_3 \in I$ and $\Gamma_+(\alpha) = \Gamma_+(I^0)$. Moreover,  $\alpha$ is strongly non-degenerate (see \cite{Fui}, Definition 4.6). By Corollary 4.9 of \cite{Fui}, $\lct(I) = \lct(I^0) =\frac{3}{2}$.

The calculation of $\lct(Jg)$ is similar. Set 
$$J = Jg = \langle 4x^3 + z^2 + 3x^2z, 2yz, y^2 +2xz + x^3\rangle .$$
Then the monomial ideal associated with the Newton polyhedron $\Gamma_+(J)$ is $J^0  =  \langle x^3, yz, y^2, z^2 \rangle$. It is obvious by formula (4.1) in \cite{Fui} that $\lct(J^0) = \frac{4}{3}$. Consider the map
$$\beta = (\beta_1, \beta_2, \beta_3)= \bigg(\frac{\partial g}{\partial x},  \frac{\partial g}{\partial y}, \frac{\partial g}{\partial z}\bigg).$$
Since $\beta_1, \beta_2, \beta_3 \in J$, $\Gamma_+(\beta) = \Gamma_+(J^0)$ and $\beta$ is strongly non-degenerate, by  Corollary 4.9 of \cite{Fui} we have $\lct(J) = \lct(J^0) =\frac{4}{3}$ 
\end{proof}

\begin{thm}\label{thm_corank4} Every germ of corank $\geq 4$ is Lipschitz modal. 
\end{thm}
\begin{proof}
By Theorem \ref{thm_deformto}, all germ of corank $\geq 4$ deform to $J_{10}$. Since $J_{10}$ are Lipschitz modal, the result follows.
\end{proof}

A direct consequence of the above theorems is
\begin{cor}
(1) A germ in $\ak m^2_n$ is Lipschitz modal if and only if it deforms to $J_{10}$.

(2) Germs in $J_{10}$ are Lipschitz modal germs of the smallest codimension.

(3) Every germ of smooth modality greater than or equal to $2$ is Lipschitz modal.
\end{cor}

\subsection*{Acknowledgements} We would like to thank  D. Trotman for suggesting the problem, and  A. Fr\"uhbis-Kr\"uger for  her useful suggestion on the proof that singularities of corank $\geq 4$ are Lipschitz modal. We would also like to thank H. D. Nguyen and A. Fernandes for their interest and helpful discussion.  The first author was supported by FAPESP grant $2016/14330 - 0$. The second author was supported by FAPESP  grant $2014/00304 - 2$ and CNPq grant $306306/2015 - 8$. The third author was supported by  FAPESP grant $2015/12667 - 5$.


\begin{thebibliography}{10}

\bibitem{Abder}
O.~M. Abderrahmane.
\newblock Poly\`edre de {N}ewton et trivialit\'e en famille.
\newblock {\em J. Math. Soc. Japan}, 54(3):513--550, 2002.

\bibitem{Arnold1}
V.~I. Arnold.
\newblock Local normal forms of functions.
\newblock {\em Invent. Math.}, 35:87--109, 1976.

\bibitem{Fui}
C.~Bivi\`a-Ausina and T.~Fukui.
\newblock Invariants for bi-{L}ipschitz equivalence of ideals.
\newblock {\em Q. J. Math.}, 68(3):791--815, 2017.

\bibitem{Brieskorn}
E.~Brieskorn.
\newblock Die {H}ierarchie der {$1$}-modularen {S}ingularit\"aten.
\newblock {\em Manuscripta Math.}, 27(2):183--219, 1979.

\bibitem{bkp}
J.~W. Bruce, N.~P. Kirk, and A.~A. du~Plessis.
\newblock Complete transversals and the classification of singularities.
\newblock {\em Nonlinearity}, 10(1):253--275, 1997.

\bibitem{Saiaetal}
J.~C.~F. Costa, M.~J. Saia, and C.~H. Soares~J\'unior.
\newblock Bi-{L}ipschitz {$\fs G$}-triviality and {N}ewton polyhedra, {$\fs
  G=f\fs R,\fs C,\fs K,\fs R_V,\fs C_V,\fs K_V$}.
\newblock In {\em Real and complex singularities}, volume 569 of {\em Contemp.
  Math.}, pages 29--43. Amer. Math. Soc., Providence, RI, 2012.

\bibitem{Ebeling}
W.~Ebeling.
\newblock {\em Functions of several complex variables and their singularities},
  volume~83 of {\em Graduate Studies in Mathematics}.
\newblock American Mathematical Society, Providence, RI, 2007.

\bibitem{fr}
A.~Fernandes and M.~A.~S. Ruas.
\newblock Bi-{L}ipschitz determinacy of quasihomogeneous germs.
\newblock {\em Glasg. Math. J.}, 46(1):77--82, 2004.

\bibitem{Humbertof}
A.~Fernandes and C.~H. Soares~J\'unior.
\newblock On the bilipschitz triviality of families of real maps.
\newblock In {\em Real and complex singularities}, volume 354 of {\em Contemp.
  Math.}, pages 95--103. Amer. Math. Soc., Providence, RI, 2004.

\bibitem{Gibson}
C.~G. Gibson.
\newblock {\em Singular points of smooth mappings}, volume~25 of {\em Research
  Notes in Mathematics}.
\newblock Pitman (Advanced Publishing Program), Boston, Mass.-London, 1979.

\bibitem{gn}
G-M Greuel and H.~D. Nguyen.
\newblock Right simple singularities in positive characteristic.
\newblock {\em J. Reine Angew. Math.}, 712:81--106, 2016.

\bibitem{hp1}
J-P. Henry and A.~Parusi\'nski.
\newblock Existence of moduli for bi-{L}ipschitz equivalence of analytic
  functions.
\newblock {\em Compositio Math.}, 136(2):217--235, 2003.

\bibitem{hp2}
J-P. Henry and A.~Parusi\'nski.
\newblock Invariants of bi-{L}ipschitz equivalence of real analytic functions.
\newblock In {\em Geometric singularity theory}, volume~65 of {\em Banach
  Center Publ.}, pages 67--75. Polish Acad. Sci. Inst. Math., Warsaw, 2004.

\bibitem{hll}
C.~T. Le and V.~C. Luong.
\newblock On tangent cones of analytic sets and {\l}ojasewicz exponents.
\newblock {\em Preprint}, 2018.

\bibitem{Milnor1}
J.~Milnor.
\newblock {\em Singular points of complex hypersurfaces}.
\newblock Annals of Mathematics Studies, No. 61. Princeton University Press,
  Princeton, N.J., 1968.

\bibitem{Mostowski}
T.~Mostowski.
\newblock Lipschitz equisingularity.
\newblock {\em Dissertationes Math. (Rozprawy Mat.)}, 243:46, 1985.

\bibitem{Duc}
H.~D. Nguyen.
\newblock Right unimodal and bimodal singularities in positive characteristic.
\newblock {\em Internat. Math. Res. Notices}, 2017.

\bibitem{nv}
N.~Nguyen and G.~Valette.
\newblock Lipschitz stratifications in o-minimal structures.
\newblock {\em Ann. Sci. \'Ec. Norm. Sup\'er. (4)}, 49(2):399--421, 2016.

\bibitem{Oka}
M.~Oka.
\newblock Non-degenerate mixed functions.
\newblock {\em Kodai Math. J.}, 33(1):1--62, 2010.

\bibitem{Parusinski}
A.~Parusi\'nski.
\newblock Lipschitz stratification of subanalytic sets.
\newblock {\em Ann. Sci. \'Ecole Norm. Sup. (4)}, 27(6):661--696, 1994.

\bibitem{Risler}
J.~J. Risler and D.~J.~A Trotman.
\newblock Bi-{L}ipschitz invariance of the multiplicity.
\newblock {\em Bull. London Math. Soc.}, 29(2):200--204, 1997.

\bibitem{rm}
M.~A.~S. Ruas and M.~J. Saia.
\newblock {$C^l$}-determinacy of weighted homogeneous germs.
\newblock {\em Hokkaido Math. J.}, 26(1):89--99, 1997.

\bibitem{Sampaio}
J.~E. Sampaio.
\newblock Bi-{L}ipschitz homeomorphic subanalytic sets have bi-{L}ipschitz
  homeomorphic tangent cones.
\newblock {\em Selecta Math. (N.S.)}, 22(2):553--559, 2016.

\bibitem{Totaro}
B.~Totaro.
\newblock The {ACC} conjecture for log canonical thresholds (after de {F}ernex,
  {E}in, {M}usta\c t\u a, {K}oll\'ar).
\newblock {\em Ast\'erisque}, (339):Exp. No. 1025, ix, 371--385, 2011.
\newblock S\'eminaire Bourbaki. Vol. 2009/2010. Expos\'es 1012--1026.

\bibitem{Whitney}
H.~Whitney.
\newblock Tangents to an analytic variety.
\newblock {\em Ann. of Math. (2)}, 81:496--549, 1965.

\end{thebibliography}
\end{document}